\documentclass{amsart}
\usepackage{amsmath}
\usepackage{amsthm}
\usepackage{amssymb}
\usepackage{enumitem}
\usepackage{graphicx}
\usepackage[hidelinks]{hyperref}
\usepackage[capitalise]{cleveref}
\usepackage{mathtools}
\usepackage{nameref}
\usepackage[new]{old-arrows}
\usepackage{stmaryrd}
\usepackage{thmtools}
\usepackage[minimal]{yhmath}

\usepackage{tikz}
\usetikzlibrary{cd}
\usetikzlibrary{positioning, angles, quotes}
\newcommand{\nospacepunct}[1]{\makebox[0pt][l]{\,#1}} 

\theoremstyle{plain}
\newtheorem{thm}{Theorem}[section]      \newtheorem*{thm*}{Theorem}
\newtheorem{cor}[thm]{Corollary}        \newtheorem*{cor*}{Corollary}
\newtheorem{prop}[thm]{Proposition}     \newtheorem*{prop*}{Proposition}
\newtheorem{lem}[thm]{Lemma}            \newtheorem*{lem*}{Lemma}
\newtheorem{claim}[thm]{Claim}          \newtheorem*{claim*}{Claim}
        \newtheorem*{exer*}{Exercise}
\newtheorem{q}[thm]{Question}           \newtheorem*{q*}{Question}
      \newtheorem*{conj*}{Conjecture}

\theoremstyle{definition}
\newtheorem{defn}[thm]{Definition}      \newtheorem*{defn*}{Definition}
\newtheorem{ex}[thm]{Example}           \newtheorem*{ex*}{Example}
\newtheorem{notation}[thm]{Notation}    \newtheorem*{notation*}{Notation}
          \newtheorem*{setup*}{Setup}

\theoremstyle{remark}
\newtheorem{rem}[thm]{Remark}           \newtheorem*{rem*}{Remark}
\newtheorem{conv}[thm]{Convention}      \newtheorem*{conv*}{Convention}

\newtheoremstyle{iremark}
    {0.5em\topsep}   
    {0.5em\topsep}   
    {\upshape}  
    {0pt}       
    {\itshape}  
    {.}         
    {5pt plus 1pt minus 1pt} 
    {\thmname{#1}\thmnumber{ \itshape#2}\thmnote{ (#3)}} 
\theoremstyle{iremark}

\theoremstyle{plain}
\newtheorem{manualtheoreminner}{Theorem}\newenvironment{manualtheorem}[1]{\renewcommand\themanualtheoreminner{#1}\manualtheoreminner}{\endmanualtheoreminner}

\newtheorem{manualcorinner}{Corollary}\newenvironment{manualcor}[1]{\renewcommand\themanualcorinner{#1}\manualcorinner}{\endmanualcorinner}

\Crefname{claim}{Claim}{Claims}
\Crefname{defn}{Definition}{Definitions}
\Crefname{ex}{Example}{Examples}
\Crefname{prop}{Proposition}{Propositions}
\Crefname{thm}{Theorem}{Theorems}
\Crefname{manualtheoreminner}{Theorem}{Theorems}

\newcommand{\ses}[3]{1 \rightarrow {#1} \rightarrow {#2} \rightarrow {#3} \rightarrow 1}

\newcommand{\FP}{\mathrm{FP}}

\newcommand{\FTP}{\mathrm{FTP}}

\newcommand{\btwo}[1]{b_{#1}^{(2)}}

\newcommand{\Dk}[1]{\mathcal D_{k{#1}}}

\newcommand{\nov}[2]{\widehat{{#1}}^{#2}}
\newcommand{\cnov}[1]{\nov{#1}{\chi}}

\newcommand{\C}{\mathbb{C}}
\newcommand{\N}{\mathbb{N}}
\newcommand{\Q}{\mathbb{Q}}
\newcommand{\R}{\mathbb{R}}
\newcommand{\Z}{\mathbb{Z}}

\newcommand{\inv}{^{-1}}

\newcommand{\ab}{^{\mathsf{ab}}}

\DeclareMathOperator{\cd}{cd}

\DeclareMathOperator{\Ext}{Ext}
\DeclareMathOperator{\Fitt}{Fitt}

\let\H\relax
\DeclareMathOperator{\H}{H}
\DeclareMathOperator{\hd}{hd}
\DeclareMathOperator{\Hom}{Hom}

\DeclareMathOperator{\Ore}{Ore}
\DeclareMathOperator{\rk}{rk}

\DeclareMathOperator{\supp}{supp}
\DeclareMathOperator{\Tor}{Tor}

\newcommand{\ol}{\mathrel{\prec}}
\newcommand{\olob}{{\ol}} 

\makeatletter
\newsavebox{\@brx}
\newcommand{\llangle}[1][]{\savebox{\@brx}{\(\m@th{#1\langle}\)}%
  \mathopen{\copy\@brx\mkern2mu\kern-0.9\wd\@brx\usebox{\@brx}}}
\newcommand{\rrangle}[1][]{\savebox{\@brx}{\(\m@th{#1\rangle}\)}%
  \mathclose{\copy\@brx\mkern2mu\kern-0.9\wd\@brx\usebox{\@brx}}}
\makeatother

\newcounter{comments}

\title{Dimension drop in residual chains}
\author{Sam P.~Fisher}
\author{Kevin Klinge}

\address[S.~P.~Fisher]{Mathematical Institute, Andrew Wiles Building, Observatory Quarter, University of Oxford, Oxford, OX2 6GG, United Kingdom}
\email{fisher@maths.ox.ac.uk}

\address[K.~Klinge]{Faculty of Mathematics, Karlsruhe Institute of Technology, Englerstrasse 2, 76131 Karlsruhe, Germany}
\email{kevin.klinge@kit.edu}

\begin{document}

\begin{abstract}
    We give a description of the Linnell division ring of a countable residually (poly-\(\Z\) virtually nilpotent) (RPVN) group in terms of a generalised Novikov ring, and show that vanishing top-degree cohomology of a finite type group \(G\) with coefficients in this Novikov ring implies the existence of a normal subgroup \(N \leqslant G\) such that \(\cd_\Q(N) < \cd_\Q(G)\) and \(G/N\) is poly-\(\Z\) virtually nilpotent.
    
    As a consequence, we show that if \(G\) is an RPVN group of finite type, then its top-degree \(\ell^2\)-Betti number vanishes if and only if there is a poly-\(\Z\) virtually nilpotent quotient \(G/N\) such that \(\cd_\Q(N) < \cd_\Q(G)\). In particular, finitely generated RPVN groups of cohomological dimension \(2\) are virtually free-by-nilpotent if and only if their second \(\ell^2\)-Betti number vanishes, and therefore \(2\)-dimensional RPVN groups with vanishing second \(\ell^2\)-Betti number are coherent. As another application, we show that if \(G\) is a finitely generated parafree group with \(\cd(G) = 2\), then \(G\) satisfies the Parafree Conjecture if and only if the terms of its lower central series are eventually free. Note that the class of RPVN groups contains all finitely generated RFRS groups and all finitely generated residually torsion-free nilpotent groups.
\end{abstract}

\maketitle

\section{Introduction}

The focus of this article is the class of residually (poly-\(\Z\) and virtually nilpotent) groups (hereafter, RPVN groups) and the \(\ell^2\)-invariants of RPVN groups. This class contains many groups of interest. It was recently shown by Kielak, Okun, Schreve, and the first author that a finitely generated group is residually finite rationally solvable (RFRS) if and only if it is residually (poly-\(\Z\) and virtually Abelian) \cite[Theorem 6.3]{OkunSchreve_DawidSimplified}, so RFRS groups are RPVN. The class of RFRS groups was introduced by Agol in connection with Thurston's Virtual Fibring Conjecture for hyperbolic \(3\)-manifolds in \cite{AgolCritVirtFib}, where he showed that right-angled Artin groups are RFRS and that the property passes to subgroups. Thus, fundamental groups of special cube complexes (introduced by Haglund--Wise \cite{HaglundWise_special}) provide a rich source of examples of RFRS groups. The \(\ell^2\)-invariants of RFRS groups are now quite well understood thanks to the work of Kielak \cite{KielakRFRS}, who gave a novel description of the Linnell division ring \(\mathcal D_{\Q G}\) of a finitely generated RFRS group \(G\) and used this to show that \(G\) virtually algebraically fibres if and only if \(\btwo{1}(G) = 0\). This was extended to higher degrees in \cite{Fisher_Improved} and the description of \(\mathcal D_{\Q G}\) was further used in \cite{Fisher_freebyZ} to show that a finite type RFRS group admits a virtual map to \(\Z\) with kernel of lower cohomological dimension if and only if its top-degree \(\ell^2\)-Betti number vanishes.

In another direction, finitely generated torsion-free nilpotent groups are poly-\(\Z\), and therefore finitely generated residually torsion-free nilpotent (RTFN) groups are RPVN as well. Central to the theory of RTFN groups is Magnus's proof that free groups are RTFN \cite{Magnus_freegrpsRTFN}, though there are many more examples. Baumslag established several properties of RTFN and provided interesting examples of RTFN groups (see, for example, \cite{Baumslag_RTFN,Baumslag_reflectionsRTFN}), including many one-relator groups. Other examples of RTFN groups are pure braid groups \cite{FalkRandell_purebraidgroupsRTFN} and pure braid groups of closed orientable surfaces of positive genus \cite{BardakovBellingeri_purebraidgroupsurfaceRTFN}, and free \(\Q\)-groups \cite{JaikinZapirain_freeQgroups}.

Note that right-angled Artin groups are also RTFN \cite{DuchampKrob_RAAGsRTFN}, and therefore so are all special groups. Despite this large intersection between the classes of RFRS and RTFN groups, neither is included in the other. For instance, a non-Abelian torsion-free nilpotent group is trivially RTFN, but it is not RFRS. On the other hand, a poly-\(\Z\) virtually Abelian group is RFRS, but it is not RTFN unless it is Abelian. Both classes are closed under free products, so there are non-amenable examples of groups that are RFRS but not RTFN, and vice versa.

\subsection{The main results}

Our first result relates the vanishing of the top-degree \(\ell^2\)-Betti number of an RPVN group to the cohomological dimension of the terms in certain residual chains of normal subgroups.

\begin{manualtheorem}{A}\label{thm:A}
    Let \(G\) be an RPVN group of type \(\FP(\Q)\) and of \(\cd_\Q(G) = d\), and let \(G = N_0 \geqslant N_1 \geqslant \dots\) be a residual chain of normal subgroups such that every quotient \(G/N_i\) is poly-\(\Z\) and virtually nilpotent. Then \(\btwo{d}(G) = 0\) if and only if \(\cd_\Q(N_i) < d\) for \(i\) sufficiently large.
\end{manualtheorem}

\begin{rem}
    This result, and the rest of the results that follow in this introduction, have analogues in positive characteristic. The field \(\Q\) will be replaced by an arbitrary field \(k\), and the \(\ell^2\)-Betti numbers will be replaced by a suitable analogue over the field \(k\).

    Note also that one direction of \cref{thm:A} follows easily from \cite[Th\'eor\`eme 6.6]{Gaboriau2002}, which states that \(\btwo{n}(G) = 0\) whenever \(G\) fits into a short exact sequence \(\ses{N}{G}{Q}\) with \(Q\) infinite amenable and \(\btwo{n}(N) < \infty\).
\end{rem}

For RTFN groups, there is a canonical residual chain of the type appearing in \cref{thm:A}. Let \(G\) be a group, let \(\gamma_0(G) = G\), and inductively define \(\gamma_{i+1}(G) = [G,\gamma_i(G)]\). Each quotient \(G/\gamma_i(G)\) is a nilpotent group, and therefore its torsion elements form a normal subgroup. Hence,
\[
    \overline{\gamma}_i(G) = \{g \in G \ : \ g^n \in \gamma_i(G) \ \text{for some} \ n > 0\}
\]
is a normal subgroup of \(G\). Then \(G\) is RTFN if and only if \(\bigcap_{i \geqslant 0} \overline{\gamma}_i(G) = \{1\}\). Our main theorem takes the following form when restricting to RTFN groups.

\begin{manualcor}{B}\label{cor:B}
    Let \(G\) be an RTFN group of type \(\FP(\Q)\) and of \(\cd_\Q(G) = d\). Then the following are equivalent:
    \begin{enumerate}
        \item \(\btwo{d}(G) = 0\);
        \item \(\cd_\Q(\gamma_i(G)) < d\) for \(i\) sufficiently large;
        \item \(\cd_\Q(\overline{\gamma}_j(G)) < d\) for \(j\) sufficiently large.
    \end{enumerate}
\end{manualcor}

\begin{rem}
    In \cref{thm:A} and \cref{cor:B}, the assumptions that \(G\) be of type \(\FP(\Q)\) and of \(\cd_\Q(G) = d\) are unnecessarily strong. We will show that both results remain true under the weaker hypotheses that \(G\) is countable and there exists a projective resolution \(0 \rightarrow P_d \rightarrow \dots \rightarrow P_0 \rightarrow \Q \rightarrow 0\) of the trivial \(\Q G\)-module \(\Q\) such that \(P_d\) is finitely generated.
\end{rem}

Recall that a group is \emph{coherent} if all its finitely generated subgroups are finitely presented. Similarly, a ring is \emph{(left/right) coherent} if and only if all its (left/right) ideals are finitely presented. Note that group rings are left coherent if and only if they are right coherent. Wise has written that coherence is essentially a phenomenon among low-dimensional groups \cite[Section 3]{Wise_anInvitation}, and has conjectured that it should be equivalent to the vanishing of the second \(\ell^2\)-Betti number (see also \cite[Section 7]{JaikinLinton_coherence} and \cite[Question 1.1]{Fisher_freebyZ} for related questions and conjectures in this direction). The following result lends further evidence to the claim that two-dimensional groups with vanishing second \(\ell^2\)-Betti number are coherent.

\begin{manualcor}{C}\label{cor:C}
    Let \(G\) be a finitely generated RPVN group of \(\cd_\Q(G) \leqslant 2\). Then \(\btwo{2}(G) = 0\) if and only \(G\) is free-by-(poly-\(\Z\) and virtually nilpotent). 
    
    In particular, if \(G\) is RPVN with \(\cd_\Q(G) \leqslant 2\) and \(\btwo{2}(G) = 0\), then \(G\) is coherent and \(kG\) is left coherent for every field \(k\).
\end{manualcor}

To obtain the conclusions about coherence, we will use the recent results obtained by Jaikin-Zapirain and Linton in \cite{JaikinLinton_coherence} to prove that free-by-polycyclic groups of rational cohomological dimension \(2\) are coherent. This extends the well-known result of Feighn--Handel that free-by-cyclic groups are coherent \cite{FeighnHandel_FreeByZCoherent}.

\subsection{Parafree groups}

A residually nilpotent group \(G\) is \emph{parafree} if there is a free group \(F\) such that \(G/\gamma_n(G) \cong F/\gamma_n(F)\) for all \(n \geqslant 0\). The study of parafree groups was initiated by H.~Neumann, who asked whether such groups are necessarily free (see \cite[I.~Introductory Remarks]{Baumslag_parafreeSurvey}). This question was answered in the negative by Baumslag, who produced many examples of (finitely generated) non-free parafree groups \cite{Baumslag_NonFreeParafree}. Baumslag also formulated many interesting conjectures on parafree groups, the most famous being the \emph{Parafree Conjecture}, which predicts that \(\H_2(G;\Z) = 0\) for all finitely generated parafree groups \(G\). Recall that the \emph{free nilpotent group} of step \(n\) on a generating set \(S\) is \(F_S/\gamma_n(F_S)\), where \(F_S\) is the free group on \(S\). Free nilpotent groups are torsion-free, so if \(G\) is a finitely generated parafree group, then \(\gamma_n(G) = \overline{\gamma}_n(G)\) for all \(n\) and therefore \(G\) is residually torsion-free nilpotent.

\begin{manualcor}{D}\label{cor:D}
    Let \(G\) be a finitely generated parafree group of cohomological dimension \(2\). Then the Parafree Conjecture holds for \(G\) if and only if \(\gamma_n(G)\) is free for all sufficiently large \(n\). In particular, if the Parafree Conjecture holds for \(G\), then \(G\) is coherent.
\end{manualcor}

\begin{rem}
    The prediction that finitely generated parafree groups \(G\) are of cohomological dimension \(2\) and satisfy \(\H_2(G;\Z) = 0\) is sometimes referred to as the \emph{Strong Parafree Conjecture}--which Baumslag apparently also believed (see \cite[p.~640]{Cochran_LinkConcordance})--and is also completely open. It is also open whether finitely generated parafree groups are finitely presentable. Note, however, that Baumslag asked whether finite presentability was a pronilpotent invariant among RTFN groups \cite[Section 12, Problem 2]{Baumslag_parafreeSurvey}, which was disproved by Bridson--Reid in \cite{BridsonReid_BaumslagProblems}.
\end{rem}

We note that along the way to proving \cref{cor:D}, we obtain the following criterion for coherence among two-dimensional RPVN groups, which may be of independent interest. 

\begin{manualtheorem}{E}\label{thm:crit_H2}
    Let \(G\) be a finitely generated RPVN of \(\cd_\Q(G) \leqslant 2\). If \(\H_2(G;\Q) = 0\), then \(G\) is coherent, and in particular is finitely presented.
\end{manualtheorem}

If \(G\) was assumed to be of type \(\FP_2(\Q)\) in the above theorem, then the result would follow immediately from \cite[Corollary 1.6]{JaikinZapirain2020THEUO}. The novelty in \cref{thm:crit_H2} actually lies in establishing a bound \(b_2(G) \geqslant \btwo{2}(G)\) for RPVN groups in the absence of finiteness conditions.

\subsection{The Linnell division ring of an RPVN group}\label{subsec:LinnellRPVN}

A group \(G\) acts on the Hilbert space of infinite square summable series
\[
    \ell^2(G) = \left\{ \sum_{g \in G} \lambda_g g \ : \ \lambda_g \in \C, \ \sum_{g \in G} |\lambda_g|^2 < \infty \right\}
\]
by left multiplication. The algebra of bounded operators commuting with this \(G\)-action is denoted by \(\mathcal N(G)\) and called the \emph{von Neumann algebra} of \(G\). The non-(zero divisors) in \(\mathcal N(G)\) form an Ore set, and the Ore localisation \(\mathcal U(G) = \Ore(\mathcal N(G))\) is called the \emph{algebra of affiliated operators} of \(G\). Note that \(\mathcal U(G)\) naturally contains a copy of \(\Q G\); the object we are interested in is the division closure of \(\Q G\) inside \(\mathcal U(G)\), which is called the \emph{Linnell ring} of \(G\) (see \cite[Chapter 8]{Luck02} for more details on this material). If \(G\) is locally indicable, then the Linnell ring of \(G\) is known to be a division ring by \cite{JaikinLopezStrongAtiyah2020} and is denoted by \(\mathcal D_{\Q G}\). In particular, the Linnell ring of an RPVN group \(G\) is a division ring containing \(\Q G\).

Our main results will depend on a description of this division ring, which is inspired by the one Kielak gave for RFRS groups in \cite{KielakRFRS}. Instead of following Kielak, however, our proof draws from the more recent description of \(\mathcal D_{\Q G}\) given by Okun and Schreve for a RFRS group \(G\) in \cite{OkunSchreve_DawidSimplified}, which elegantly eliminates many of the technical difficulties encountered in Kielak's description. In \cite{KielakRFRS} and \cite{OkunSchreve_DawidSimplified}, the Linnell division ring \(\mathcal D_{\Q G}\) of a RFRS group \(G\) is described as a union of finite-dimensional extensions of Novikov rings of finite-index subgroups of \(G\), where the \emph{Novikov ring} \(\cnov{\Q \Gamma}\) of a group \(\Gamma\) and a character \(\chi \colon \Gamma \rightarrow \R\) is the ring of power series of \(\Gamma\) which tend to \(+\infty\) with respect to \(\chi\) (see \cref{def:novikov1}).

To obtain an analogous description of \(\mathcal D_{\Q G}\) for an RPVN group \(G\), we introduce the concept of a \emph{multicharacter} on a group \(\Gamma\) with a nilpotent quotient \(Q\). A character on a nilpotent group \(Q\) contains the derived subgroup in its kernel, so it will by necessity destroy much of the information contained in \(Q\). A multicharacter is by definition a finite set of characters, one for each successive quotient in a central series for \(Q\), which has the advantage of retaining the nilpotent structure of \(Q\). Associated to a multicharacter \(\chi\), we define a Novikov ring which we also denote by \(\cnov{\Q \Gamma}\) (see \cref{def:novikov_ring}), and obtain a similar description of \(\mathcal D_{\Q G}\) for an RPVN group \(G\) in terms of Novikov rings of its finite-index subgroups (see \cref{thm:divring_elements}). 

\begin{rem}
    If all of the characters in a multicharacter on \(Q\) are injective, then the multicharacter induces an order \(\ol\) on \(Q\), and the Novikov ring \(\cnov{\Q \Gamma}\) coincides with the Malcev--Neumann ring \(\Q N *_{\ol} Q\) of power series with well-ordered support (here \(N = \ker(\Gamma \rightarrow Q)\)).
\end{rem}

The description of \(\mathcal D_{\Q G}\) is then used to relate the \(\ell^2\)-homology of an RPVN group \(G\) with its Novikov cohomology, which we in turn show is connected to the cohomological dimension of conilpotent subgroups via the following result, which generalises \cite[Theorem 3.5]{Fisher_freebyZ}.

\begin{manualtheorem}{F}\label{thm:F}
    Let \(G\) be a group of type \(\FP(\Q)\) with \(\cd_\Q(G) = d\) and suppose that \(\chi = (\chi_1, \dots, \chi_n)\) is a multicharacter on a nilpotent quotient \(G/N\). If
    \[
        \H^d(G; \nov{\Q G}{(\pm\chi_1, \dots, \pm\chi_n)}) = 0
    \]
    for all choices of signs, then \(\cd_\Q(N) < \cd_\Q(G)\).
\end{manualtheorem}

\subsection{Acknowledgments}

The first author would like to thank his supervisor, Dawid Kielak, for the many discussions and comments on this work, as well as Marco Linton, Ismael Morales, and Pablo Sánchez-Peralta for enlightening conversations. The first author is supported by the National Science and Engineering Research Council (NSERC) [ref.~no.~567804-2022] and the European Research Council (ERC) under the European Union's Horizon 2020 research and innovation programme (Grant agreement No. 850930).
The second author is grateful to Claudio Llosa Isenrich and Roman Sauer for many helpful discussions.
The second author is funded by the German Research Foundation (Deutsche Forschungsgemeinschaft) via project 281869850 and RTG 2229 ``Asymptotic Invariants and Limits of Groups and Spaces''.

\section{Preliminaries}

\subsection{Orders on nilpotent groups}

We recall some basics on nilpotent groups that will be used throughout the article. A group \(Q\) is \emph{nilpotent} if there is an integer \(n\) and a series \(1 = Q_n \leqslant Q_{n-1} \leqslant \dots \leqslant Q_0 = Q\) of normal subgroups of \(Q\) such that \(Z(Q_{i-1}/Q_i) \leqslant Z(Q/Q_i)\) for each \(1 \leqslant i \leqslant n\), where
\[
    Z(\Gamma) = \{g \in \Gamma \ : \ gh = hg \ \text{for all} \ h \in \Gamma \}
\]
denote the \emph{centre} of a group \(\Gamma\). The quotient \(Q_{i-1}/Q_i\) are called the \emph{successive quotients} of the central series, and \(n\) is called the \emph{length} of the central series. The minimal length \(n\) of a central series for \(Q\) is called the \emph{nilpotency class} of \(Q\), and \(Q\) is then called a nilpotent group of \emph{step} \(n\). Thus, a group is nilpotent of step \(1\) if and only if it is Abelian.

Recall that a group \(Q\) is \emph{orderable} if it admits a total order \(\ol\) such that \(g \ol h\) if and only if \(tg \ol th\) if and only if \(gt \ol ht\). This is sometimes called \emph{bi-orderability}, but since we will not deal with left or right orders on groups, we prefer this terminology. If \(Q\) is torsion-free and nilpotent, then it is not hard to show that it has a central series with torsion-free (Abelian) successive quotients (in fact, the upper central series will do for this). The following lemma is well known; since it plays an important role in what follows, we include a proof for convenience

\begin{lem}
    Torsion-free nilpotent groups are orderable.
\end{lem}
\begin{proof}
    Let \(Q\) be torsion-free nilpotent of step \(n\) and let \(1 = Q_n \leqslant \dots \leqslant Q_0 = Q\) be a central series with torsion-free successive quotients. Each quotient \(Q_{i+1}/Q_i\) is torsion-free Abelian, and hence orderable; this is because a finitely generated Abelian group embeds into \(\R\) and hence is orderable, and locally orderable groups are orderable. Let \(\ol_{i-1}\) be an order on \(Q_{i-1}/Q_i\). For nontrivial \(q \in Q\), declare \(1 \ol q\) if and only if \(1 \ol_{i-1} qQ_i\), where \(i\) is maximal such that \(q \in Q_{i-1}\). This defines an order on \(Q\). \qedhere
\end{proof}

Orders obtained as in the proof of the previous lemma are called \emph{lexicographic} with respect to the central series.

\begin{lem}[{\cite[p.~227]{BludovGlassRhem_centralConvex}}]\label{lem:all_orders_are_lex}
    Every order on a torsion-free nilpotent group \(Q\) is lexicographic with respect to some central series.
\end{lem}

\subsection{Cohomological dimension and finiteness properties}\label{subsec:finprops}

Let \(G\) be a group and \(R\) be a nonzero unital ring. The \emph{cohomological dimension} of \(G\) over \(R\) is 
\[
    \cd_R(G) = \sup\{n : \H^n(G;M) \neq 0 \ \text{and} \ M \ \text{is an} \ RG\text{-module}\}.
\]
The group \(G\) is of \emph{type \(\FP_n(R)\)} if there is a projective resolution 
\[
    \cdots \rightarrow P_1 \rightarrow P_0 \rightarrow R \rightarrow 0
\]
of the trivial \(RG\)-module \(R\), where \(P_i\) is finitely generated for all \(i \leqslant n\). If all the modules \(P_i\) are finitely generated, then \(G\) is of \emph{type \(\FP_\infty(R)\)}. If, additionally, \(P_i = 0\) for all sufficiently large \(i\), then \(G\) is of type \(\FP(R)\). Note that \(G\) is finitely generated if and only if \(G\) is of type \(\FP_1(R)\) for some ring \(R\) \cite[Proposition 2.1]{BieriQueenMary}. If \(G\) is of type \(\FP_n(\Z)\) (resp.~\(\FP_\infty(\Z)\), resp.~\(\FP(\Z)\), then it is of type \(\FP_n(R)\) (resp.~\(\FP_\infty(R)\), resp.~\(\FP(R)\)) for all rings \(R\). Finally, if \(G\) admits a classifying space with finitely many cells of dimension at most \(n\) (resp.~cells in each dimension, resp.~cells), then \(G\) is of type \(\FP_n(R)\), (resp.~\(\FP_\infty(R)\), resp.~\(\FP(R)\)) for all rings \(R\).

We will often be interested in groups that have the following finiteness property. We say that a group \(G\) is of \emph{type \(\FTP_n(R)\)} if the trivial \(RG\)-module \(R\) admits a projective resolution of the form
\[
    0 \rightarrow P_n \rightarrow \dots \rightarrow P_0 \rightarrow R \rightarrow 0
\]
where \(P_n\) is finitely generated. If \(G\) is of type \(\FP_n(R)\) and \(\cd_R(G) \leqslant n\), then \(G\) is of type \(\FTP_n(R)\). If \(G\) is of type \(\FTP_n(R)\), then \(\cd_R(G) \leqslant n\). The terminology \(\FTP_n(R)\) is non-standard.

\subsection{The category of \texorpdfstring{\(R\)}{R}-division rings and specialisations}\label{subsec:Rdivrings}

The material of this subsection will be used in the proof of \cref{cor:D}, but no other results of the article will depend on it. Let \(R\) be an associative ring with unity. An \emph{\(R\)-division ring} is a division ring \(\mathcal D\) and a map \(\varphi \colon R \rightarrow \mathcal D\), such that \(\varphi(R)\) generates \(\mathcal D\) as a division ring (i.e.~\(\mathcal D\) is the smallest division subring of \(\mathcal D\) containing \(\varphi(R)\)).

A \emph{specialisation} of \(R\)-division rings \(\varphi_1 \colon R \rightarrow \mathcal D_1\) and \(\varphi_2 \colon R \rightarrow \mathcal D_2\) is an epimorphism of \(R\)-rings \(\rho \colon D \rightarrow \mathcal D_2\), where \(D \leqslant \mathcal D_1\) is a subring containing \(\varphi(R)\) such that every element of \(D \smallsetminus \ker \rho\) is invertible in \(D\) and the diagram
\[
    \begin{tikzcd}
        & R \arrow[ld, "\varphi_1"'] \arrow[d] \arrow[rd, "\varphi_2"] & \\
        \mathcal D_1 & D \arrow[l, hook'] \arrow[r,"\rho"] & \mathcal D_2
    \end{tikzcd}
\]
commutes. Note that \(D\) is thus a local ring with maximal ideal \(\ker \rho\). We will denote the specialisation by \(\rho \colon \mathcal D_1 \rightarrow \mathcal D_2\), even though \(\rho\) is not a map between the rings \(\mathcal D_i\).

We say that \(R\)-division rings \(\mathcal D_1\) and \(\mathcal D_2\) are \emph{isomorphic} if there is a bijection \(\mathcal D_1 \rightarrow \mathcal D_2\) of \(R\)-rings. This is in fact equivalent to the existence of two specialisations \(\mathcal D_1 \rightarrow \mathcal D_2\) and \(\mathcal D_2 \rightarrow \mathcal D_1\). Hence, the category of \(R\)-division rings and specialisations forms a poset, where \(\mathcal D_1 \geqslant \mathcal D_2\) if and only if there is a specialisation \(\mathcal D_1 \rightarrow \mathcal D_2\). An \(R\)-division ring \(\mathcal D\) is \emph{universal} if \(\mathcal D \geqslant \mathcal E\) for all \(R\)-division rings \(\mathcal E\).

Note that every \(R\)-division ring determines a rank function on the matrices over \(R\). Indeed, if \(\varphi \colon R \rightarrow \mathcal D\) is an \(R\)-division ring and \(M \colon R^n \rightarrow R^m\) is a matrix over \(R\), then there is an induced matrix \(M_{\mathcal D} \colon \mathcal D^n \rightarrow \mathcal D^m\), obtained by applying the functor \(\mathcal D \otimes_R -\) to \(M\). The \emph{rank} of \(M\) determined by \(\mathcal D\) is the dimension of the image of \(M_\mathcal D\) as a \(\mathcal D\)-module, and it is denoted by \(\rk_{\mathcal D}(M)\). Given two \(R\)-division rings \(\mathcal D_1\) and \(\mathcal D_2\), we write \(\rk_{\mathcal D_1} \geqslant \rk_{\mathcal D_2}\) if and only if \(\rk_{\mathcal D_1}(M) \geqslant \rk_{\mathcal D_2}(M)\) for every matrix \(M\) over \(R\). In this way, the \(R\)-division rings induce another poset, which turns out to be isomorphic to the poset of \(R\)-division rings and specialisations by the following result.

\begin{thm}[{\cite[Theorem 2]{Malcolmson_SkewFields}}]\label{thm:Malcolmson}
    If \(\mathcal D_1\) and \(\mathcal D_2\) are \(R\)-division rings, then \(\mathcal D_1 \geqslant \mathcal D_2\) if and only if \(\rk_{\mathcal D_1} \geqslant \rk_{\mathcal D_2}\).
\end{thm}

The theory of \(R\)-division rings and specialisations is due to Cohn, and we refer the reader to his book \cite[Chapter 7]{Cohn_FreeRingsRelations} for more details on the subject.

\section{Novikov rings and division rings}

\subsection{Novikov rings}

\begin{defn}[Crossed products]
    Let \(R\) be a ring and let \(G\) be a group. A \emph{crossed product} \(R*G\) is a ring with underlying Abelian group \(\bigoplus_{g \in G} R_g\), where \(R_g \subseteq R\) are additive subgroups, \(R_g R_h \subseteq R_{gh}\) for all \(g,h \in G\), each summand \(R_g\) contains a unit \(u_g\), and \(R_e = R\).
\end{defn}

If \(R*G\) is a crossed product, then every element of \(x \in R*G\) is uniquely expressed as a finite sum \(\sum_{g \in G} r_g u_g\), where \(r_g \in R\).

\begin{ex}
    The basic example of a crossed product is the group ring \(RG\) of a group \(G\) with coefficients in the ring \(R\), which we call a \emph{trivial} crossed product. Non-trivial examples often arise in the following setting. Suppose \(\ses{N}{G}{Q}\) is an extension of groups and that \(R\) is a ring. Then the group ring \(RG\) is isomorphic to a crossed product \(RN * Q\). More generally, a twisted group ring \(R*G\) decomposes as \((R*N)*Q\) in this setting. Note that the isomorphism is non-canonical as it depends on a choice of set-theoretic section \(Q \rightarrow G\).
\end{ex}

We will also frequently use infinite series of group elements, which we formalise as follows.

\begin{defn}[Twisted formal series]
    If \(R * G\) is a crossed product and \(M\) is an \(R\)-\(RG\)-bimodule, then we can form the \(R * G\)-bimodule \(M * \llbracket G \rrbracket\) of infinite formal sums of elements of \(M\) indexed by \(G\).
    
    More precisely, denote the action of \(G\) on \(M\) by \(m \mapsto m^s\) for \(s \in G\) and denote the distinguished units of \(R * G\) by \(u_s\) for \(s \in G\). The elements of \(M * \llbracket G \rrbracket\) are denoted by formal series \(\sum_{g \in G} m_g u_g\). Then the \(R*G\)-bimodule structure on \(M * \llbracket G \rrbracket\) is given by
    \[
        r u_s \cdot \sum_{g \in G} m_g u_g := \sum_{g \in G} \left(r m_g^s (u_s u_g u_{sg}\inv) \right) u_{sg}
    \]
    \[
        \sum_{g \in G} m_g u_g \cdot r u_s = \sum_{g \in G} \left( m_g (u_g r u_su_{gs}\inv) \right) u_{gs},
    \]
    where \(r \in R\).

    The elements \(m_g \in M\) appearing in a formal series \(x = \sum_{g \in G} m_g u_g\) are the \emph{coefficients} of \(x\) and the set \(\supp(x) := \{g \in G : m_g \neq 0\}\) is the \emph{support} of \(x\).
\end{defn}

The basic example of the construction above is the following. 

\begin{ex}
    If \(R*G\) is a crossed product, then \(G\) acts on \(R\) by \(r \mapsto r^g = u_g r u_g\inv\), so we may form \(R * \llbracket G \rrbracket\). Note that \(R * \llbracket G \rrbracket\) does not have a natural ring structure, unless \(G\) is finite, in which case it coincides with \(R*G\).
\end{ex}

\begin{defn}[Novikov rings]\label{def:novikov1}
    Let \(R*G\) be a crossed product of a group \(G\) and a ring \(R\) and let \(\chi \colon G \rightarrow \R\) be homomorphism. The \emph{Novikov ring} \(\cnov{R*G} \subseteq R * \llbracket G \rrbracket\) is the subset
    \[
        \left\{ x = \sum_{g \in G} r_g u_g : r_g \in R, \ \#\left( \supp(x) \cap \chi\inv \left(\left] -\infty, t \right]\right) \right) < \infty \ \text{for all} \ t \in \R \right\}.
    \]
\end{defn}

It is not hard to show that the map \(\cnov{R*G} \times \cnov{R*G} \rightarrow \cnov{R*G}\) given by 
\[
    \sum_{g \in G} r_g u_g \cdot \sum_{h \in G} r_h' u_h \quad \mapsto \quad \sum_{s \in G} \sum_{\substack{(g,h) \in G^2 \\ gh = s}} (r_g u_g r_g' u_h u_s\inv) u_s
\]
is well-defined and endows \(\cnov{R*G}\) with the structure of a ring.

\begin{lem}\label{lem:extendCrossProd}
    Let \(\ses{N}{G}{Q}\) be an extension of groups, let \(\chi \colon N \rightarrow \R\) be a homomorphism, and suppose that \(\chi(n) = \chi(n^g)\) for every \(n \in N\) and every \(g \in G\). If \(R*G\) is a crossed product, then there is a crossed product \(\cnov{R*N}*Q\), which naturally contains the crossed product \(R*G \cong (R*N)*Q\).
\end{lem}

\begin{proof}
    Fix a section \(s \colon Q \rightarrow G\). If \(x = \sum_{n \in N} r_n u_n\) is contained in \(\cnov{R*N}\), then so is
    \begin{align*}
        u_{s(q)} x u_{s(q)}\inv &= \sum_{n \in N} u_{s(q)} r_n u_n u_{s(q)}\inv \\
        &= \sum_{n \in N} \left( u_{s(q)} r_n u_{s(q)}\inv u_{s(q)} u_n u_{s(q)}\inv u_{s(q)ns(q)\inv}\inv \right) u_{s(q)ns(q)\inv},
    \end{align*}
    which follows immediately from the assumption that \(\chi(n) = \chi(s(q)ns(q)\inv)\). Hence, the multiplication on \(R*G \cong (R*N)*Q\) extends naturally to a well-defined multiplication on \(\cnov{R*N}*Q\). \qedhere
\end{proof}

In particular, \cref{lem:extendCrossProd} applies to central extensions \(\ses{N}{G}{Q}\).

\begin{defn}[Multicharacters and Novikov rings]\label{def:novikov_ring}
    Let \(Q\) be a nilpotent group. A \emph{multicharacter} \(\chi\) on \(Q\) is a list of homomorphisms
    \[
        \chi = (\chi_i \colon Q_i/Q_{i+1} \rightarrow \R)_{i=0}^{n-1},
    \]
    where \(1 = Q_n \leqslant \dots \leqslant Q_0 = Q\) is a central series. The homomorphisms \(\chi_i\) are called the \emph{components} of \(\chi\), and the central series and \(\chi\) are \emph{associated} to each other.

    Let \(R\) be a ring and let \(\chi\) be a multicharacter on \(Q\) associated to a central series as above, and suppose that \(R * Q\) is a crossed product. By \cref{lem:extendCrossProd}, we may define the rings
    \[
        R_n = R, \quad R_i = \nov{R_{i+1} * Q_i/Q_{i+1}}{\chi_i} \ \text{for} \ n > i \geqslant 0
    \]
    by reverse induction. We call \(R_0\) the \emph{Novikov ring} of \(R*Q\) and \(\chi\), and denote it by \(\nov{R * Q}{\chi}\).

    Suppose that \(\ses{N}{G}{Q}\) is a short exact sequence of groups, where \(Q\) is nilpotent, and \(\chi\) is a multicharacter on \(Q\). Then we call \(\nov{RG}{\chi} := \nov{RN*Q}{\chi}\) the \emph{Novikov ring} of \(RG\) and \(\chi\).
\end{defn}

\begin{rem}
    If \(Q\) is nilpotent of step \(n\), then \(\chi\) must have at least \(n\) components. In particular, if \(\chi\) has one component, then \(Q\) is Abelian and \(\nov{RG}{\chi}\) coincides with the usual Novikov ring of \cref{def:novikov1}.
\end{rem}

\begin{lem}\label{lem:novikov_series}
    Let \(Q\) be a nilpotent group and let \(R*Q\) be a crossed product. Let \(\chi\) be a multicharacter associated to a central series \(1 = Q_n \leqslant \dots \leqslant Q_0 = Q\). Then there is a natural inclusion \(\cnov{R*Q} \hookrightarrow R*\llbracket Q \rrbracket\) of \(R*Q\)-bimodules.
\end{lem}
\begin{proof}
    Consider the modules \(M_n = R\) and \(M_i = M_{i+1} * \llbracket Q_i/Q_{i+1} \rrbracket\), for \(i \leqslant n\). Then \(R * \llbracket Q \rrbracket \cong M_0\), and it follows by reverse induction that there are \(R*Q_i\)-bimodule inclusions \(R_i \hookrightarrow M_i\) for each \(i \leqslant n\) (where we are using the notation of \cref{def:novikov_ring}).
\end{proof}

\subsection{Division rings}

In this subsection, we will introduce the various division rings that we will work with, and the relations between them. Note that these division rings will always contain the group algebra \(kG\) as subrings, so the groups we study necessarily have no zero divisors, and in particular are torsion-free. There are no known examples of torsion-free groups whose group algebras do not admit an embedding into a division ring, and there are several conjectures that predict that group algebras of (some subclasses) of torsion-free groups admit embeddings into division rings. For example, \emph{Malcev's Problem} is to show that \(kG\) admits an embedding into a division ring whenever \(G\) is left orderable. The \emph{Strong Atiyah Conjecture} for a torsion-free group \(G\) is equivalent to the Linnell ring being a division ring (see \cref{subsec:LinnellRPVN}), and Jaikin-Zapirain and Linton conjecture that \(kG\) admits a unique embedding into a division ring that has the Linnell property (see \cref{rem:Linnell}).

We begin by recalling the Ore condition.

\begin{defn}[Ore condition]
    Let \(R\) be a ring. A subset \(S \subseteq R\) is called \emph{multiplicative} if \(1 \in S\) and \(s_1, s_2 \in S\) implies \(s_1s_2 \in S\). We say that \(R\) satisfies the \emph{(right) Ore condition} at a multiplicative subset \(S \subseteq R\) if \(rS \cap sR \neq \varnothing\) for all \(s \in S\) and \(r \in R\). If \(R\) satisfies the Ore condition at \(S\), then we may form a ring \(\Ore_S(R)\) called the \emph{Ore localisation} of \(R\) at \(S\), which consists of formal expressions \(rs\inv\), where \(r \in R\), \(s \in S\). The expressions \(rs\inv\) are equivalence classes of pairs \((r,s) \in R \times S\), where \((r,s) \sim (r',s')\) if and only if there is some \(s_0 \in S\) such that
    \[
        rs_0 = r' \quad \text{and} \quad ss_0 = s'.
    \]
    The interested reader can find more details on the Ore condition in \cite[Section 4.4]{PassmanGrpRng}.

    If \(R\) has no zero divisors and it satisfies the Ore condition at \(R \smallsetminus \{0\}\), then \(R\) is called an \emph{Ore domain} and its Ore localisation at \(R \smallsetminus \{0\}\) is a division ring denoted \(\Ore(R)\).
\end{defn}

If \(R\) is an Ore domain, then the map \(i \colon R \rightarrow \Ore(R)\) given by \(i(r) = r1\inv\) is an embedding, and it satisfies the following universal property: If \(f \colon R \rightarrow D\) is a ring homomorphism such that \(f(R \smallsetminus \{0\}) \subseteq D^\times\), then there is a unique map \(\Ore(f) \colon \Ore(R) \rightarrow D\) such that \(f = \Ore(f) \circ i\).

The following instance of the above fact is useful. If \(R \subseteq D\) is an inclusion of rings, where \(R\) is an Ore domain and \(D\) is a division ring such that there are no division rings \(E\) satisfying \(R \subset E \subset D\), then \(D \cong \Ore(R)\). We will also make frequent use of the following result (recall that nilpotent groups are elementary amenable).

\begin{thm}[{Kropholler--Linnell--Moody \cite{KrophollerLinnellMoody_Ore}}]
    If $k*Q$ is a crossed product of a division ring $k$ and a torsion-free elementary amenable group $Q$, then $k*Q$ is an Ore domain.
\end{thm}

\begin{defn}[Hughes-free embeddings] \label{def:hughes_free}
    Let \(G\) be a locally indicable group and let \(k\) be a field. If \(\mathcal D\) is a division ring and there is an embedding \(\varphi \colon kG \hookrightarrow \mathcal D\) such that
    \begin{enumerate}
        \item \(\varphi(kG)\) generates \(\mathcal D\) as a division ring, and
        \item for every pair of subgroups \(N \lhd H\) of \(G\) such that \(H/N \cong \Z\), the multiplication map
        \[
            \mathcal D_N \otimes_{kN} kH \rightarrow \mathcal D
        \]
        is injective, where \(\mathcal D_N\) denotes the division closure of \(kN\) in \(\mathcal D\),
    \end{enumerate}
    then the embedding \(\varphi\) is \emph{Hughes-free}, and we say that \(\mathcal D\) is a \emph{Hughes-free} division ring for \(kG\).
\end{defn}

\begin{rem}\label{rem:Linnell}
    There is also the following related more general notion: Let \(G\) be a group such that \(\varphi \colon kG \hookrightarrow \mathcal D\) is an embedding, where \(\mathcal D\) is a division ring. Then \(\varphi\) is \emph{Linnell} (following the terminology of \cite{JaikinLinton_coherence}) if the multiplication map 
    \[
        \mathcal D_H \otimes_{kH} \mathcal D \rightarrow \mathcal D
    \]
    is injective for any subgroup \(H \leqslant G\). It is not hard to show that if \(kG\) is an Ore domain, then \(kG \hookrightarrow \mathcal \Ore(kG)\) is Linnell. By a recent result of Gräter, if \(G\) is locally indicable and \(kG \hookrightarrow \mathcal D\) is Hughes-free, then it is Linnell \cite[Corollary 8.3]{Grater20}.
\end{rem}

\begin{thm}[{Hughes \cite{HughesDivRings1970}}]
    Let \(G\) be a locally indicable group and let \(k\) be a field. If \(kG\) has a Hughes-free embedding, then it is unique up to (\(kG\)-algebra)-isomorphism.
\end{thm}

In view of this result, if \(\varphi \colon kG \hookrightarrow \mathcal D\) is a Hughes-free embedding, then we will denote \(\mathcal D\) by \(\Dk{G}\) and call it the Hughes-free division ring of \(kG\). We will also suppress the embedding \(\varphi\) and simply view \(kG\) as a subring of \(\Dk{G}\). Note that the division closure of \(kH\) in \(\Dk{G}\) is isomorphic to \(\Dk{H}\) for any subgroup \(H \leqslant G\).

\begin{defn}[Malcev--Neumann division ring]\label{def:malcev_neumann}
    Let \(G\) be a group with an order \(\ol\), let \(k\) be a division ring, and let \(k*G\) be a crossed product. The \emph{Malcev-- Neumann division ring} of \(k*G\) with respect to \(\ol\) is
    \[
        k *_\prec G = \left\{ \sum_{g \in G} \lambda_g u_g \ : \ \lambda_g \in k, \ \{g : \lambda_g \neq 0\} \ \text{is well-ordered with respect to} \ \prec \right\}.
    \]
\end{defn}

Malcev \cite{Malcev_series} and Neumann \cite{Neumann_series} independently proved that \(k *_\prec G\) is a division ring whose addition and multiplication extend those of \(k*G\) in the natural way.

The following proposition relates the constructions introduced so far. Its proof is standard; the interested reader may wish to consult \cite[Lemmas 3.1 and 3.2]{OkunSchreve_DawidSimplified}.

\begin{prop}\label{prop:ore_HF_MN}
    Let \(\ses{N}{G}{Q}\) be a short exact sequence of groups where \(G\) is locally indicable, and let \(k\) be a field such that \(\Dk{G}\) exists. Then there is a crossed product structure \(\Dk{N} * Q\) such that \(kG \cong kN * Q \subseteq \Dk{N} * Q \subseteq \Dk{G}\) and the following hold.
    \begin{enumerate}
        \item If \(Q\) is amenable, then \(\Dk{G} \cong \Ore(\Dk{N} * Q)\).
        \item If \(Q\) has an order \(\ol\), then \(\Dk{G}\) is isomorphic to the division closure of \(kG\) in \(\Dk{N} *_\prec Q\).
    \end{enumerate}
\end{prop}

Thus, a short exact sequence \(\ses{N}{G}{Q}\), where \(G\) is locally indicable and \(\Dk{G}\) exists, and \(Q\) is amenable and has an order \(\ol\) determines an embedding \(\iota_\prec \colon \Dk{G} \hookrightarrow \Dk{N} *\prec Q\). In this way, we can represent the elements of \(\Dk{G}\) as power series over \(Q\) with coefficients in \(\Dk{N}\).

\subsection{Summary}

We want to relate the constructions of the previous two subsections. Let \(k\) be a field and \(\ses{N}{G}{Q}\) be a short exact sequence of groups such that \(G\) is locally indicable, \(Q\) is torsion-free nilpotent, and \(\Dk{G}\) exists. Moreover, let \(\chi\) be a multicharacter on \(Q\), and \(\ol\) be an order on \(Q\). By \cref{lem:novikov_series,prop:ore_HF_MN}, there are inclusions
\[
    \Dk{G} \xhookrightarrow{\iota_\prec} \Dk{N} *_\prec Q \hookrightarrow \Dk{N} * \llbracket Q \rrbracket \hookleftarrow \cnov{\Dk{N} * Q}
\]
of \(\Dk{N}*Q\)-bimodules, which fit into the commutative diagram of \(kG\)-modules in \cref{fig:big_diagram}. All the maps in the diagram that do not have codomain \(\Dk{N}*\llbracket Q\rrbracket\) are ring homomorphisms. Thus, when an order \(\prec\) on \(Q\) is fixed, we will think of all the rings appearing above as subset of the common environment \(\Dk{N} * \llbracket Q \rrbracket\). For example, it makes sense to ask whether an element of \(\Dk{G}\) lies in \(\cnov{kG}\), once an order on \(Q\) is fixed.

\begin{figure}[ht]
    \[
    \begin{tikzcd}
        {kG} \arrow[r, "\cong", hook] & {kN*Q} \arrow[d, hook] \arrow[r, hook] & \cnov{kG} = \cnov{kN * Q} \arrow[r,hook] & \cnov{\Dk{N}*Q} \arrow[dd, hook] \\
        & \Dk{N}*Q \arrow[d, hook] & & \\
        \Dk{G} \arrow[r, "\cong", hook] \arrow[rr, "\iota_\prec"', bend right=25] & \Ore(\Dk{N}*Q) \arrow[r, hook] & \Dk{N}*_\prec Q \arrow[r, hook] & \Dk{N}*\llbracket Q\rrbracket
    \end{tikzcd}
    \]
    \caption{A commutative diagram summarising the various inclusions between the rings and modules introduced above.}
    \label{fig:big_diagram}
\end{figure}

\subsection{\texorpdfstring{\(\Dk{G}\)}{DkG}-Betti numbers}

Hughes-free embeddings \(kG \hookrightarrow \Dk{G}\) give rise to powerful homological invariants of \(G\). Indeed, in this situation, \(\Dk{G}\) becomes a \(kG\)-bimodule, and we can compute the group homology and cohomology
\[
    \H_n(G; \Dk{G}) = \Tor_n^{kG}(\Dk{G}, k) \qquad \text{and} \qquad \H^n(G; \Dk{G}) = \Ext_{kG}^n(k, \Dk{G}).
\]
Since \(\Dk{G}\) is a division ring, any module over \(\Dk{G}\) is free and has a well-defined dimension (by the same argument that vector spaces have well-defined dimensions). Thus, we can define the \(\Dk{G}\)-Betti numbers
\[
    \btwo{n}(G;k) = \dim_{\Dk{G}} \H_n(G; \Dk{G}) \qquad \text{and} \qquad b_{(2)}^n(G;k) = \dim_{\Dk{G}} \H^n(G; \Dk{G}).
\]
The Strong Atiyah Conjecture (over \(\C\)) holds for all locally indicable groups by \cite{JaikinZapirain2020THEUO}, so \(\mathcal D_{\Q G}\) exists for all locally indicable groups \(G\) and coincides with the Linnell division ring of \(G\) by \cite{LinnellDivRings93}. Thus, \(\btwo{n}(G;\Q)\) (resp.~\(b_{(2)}^n(G;\Q)\)) equals the usual \(\ell^2\)-Betti number \(\btwo{n}(G)\) (resp.~\(b_{(2)}^n(G)\)). When \(k\) is of positive characteristic, we view the invariants \(\btwo{n}(G;k)\) as positive characteristic analogues of \(\ell^2\)-Betti numbers.

The following properties will be useful.

\begin{prop}
    Let \(G\) be a locally indicable group and let \(k\) be a field such that \(\Dk{G}\) exists.
    \begin{enumerate}[label=(\arabic*)]
        \item\label{item:scale} If \(H \leqslant G\) is a subgroup of finite index, then \(\btwo{n}(H;k) = [G:H] \cdot \btwo{n}(G;k)\) for all \(n\).
        \item\label{item:eulerChar} If the trivial \(kG\)-module \(k\) admits a finite \emph{free} resolution by finitely generated \(kG\)-modules, then \(\chi(G) = \sum_{n \geqslant 0} (-1)^n \btwo{n}(G;k)\).
        \item\label{item:homol_equals_cohomol} \(\btwo{n}(G;k) = b_{(2)}^n(G;k)\) for all \(n\), provided at least one of the quantities is finite.
    \end{enumerate}
\end{prop}
\begin{proof}
    \cref{item:scale} follows easily from the fact that \(kG \hookrightarrow \Dk{G}\) is a Linnell embedding \cite[Corollary 8.3]{Grater20} (see \cite[Lemma 6.3]{Fisher_Improved} for a detailed proof). \cref{item:eulerChar} is a standard argument in homological algebra. Similarly, a standard homological algebra argument shows that \(\H^n(G;\Dk{G}) = \Hom_{\Dk{G}}(\H_n(G; \Dk{G}), \Dk{G})\), whence \cref{item:homol_equals_cohomol} follows. \qedhere
\end{proof}

We conclude this section with a division ring version of Gaboriau's theorem \cite[Th\'eor\`eme 6.6]{Gaboriau2002}, which shows that the invariants \(\btwo{n}(G;k)\) are obstructions to fibring and to the existence of low cohomological dimension infinite-index normal subgroups \(N\) with amenable quotients \(G/N\). Note that the result is proved in \cite[Theorem 6.4]{Fisher_Improved} in the case where the quotient is \(\Z\); the proof is similar.

\begin{prop}\label{prop:divring_gaboriau}
    Let \(\ses{N}{G}{Q}\) be a short exact sequence, where \(G\) is locally indicable, \(Q\) is infinite and amenable, and \(k\) is a field such that \(\Dk{G}\) exists. If \(\btwo{n}(N;k) < \infty\), then \(\btwo{n}(G;k) = 0\).
\end{prop}
\begin{proof}
    Let \(C_\bullet \rightarrow k \rightarrow 0\) be a free resolution by \(kG\)-modules. Since \(\Dk{G}\) is Linnell, we have identifications \(\Dk{G} = \Ore(\Dk{N} * Q)\) and
    \[
        \Dk{N} \otimes_{kN} C_n \cong \Dk{N} \otimes_{kN} \bigoplus_{I_n} kG \cong \bigoplus_{I_n} \Dk{N} \otimes_{kN} kG \cong \bigoplus_{I_n} \Dk{N} * Q
    \]
    for some index set \(I_n\), for each integer \(n\). Thus, we have inclusions of chain complexes
    \[
        \begin{tikzcd}
            \cdots \arrow[r] & \bigoplus_{I_{n+1}} \Dk{N} * Q \arrow[d, hook] \arrow[r] & \bigoplus_{I_n} \Dk{N} * Q \arrow[d, hook] \arrow[r] & \bigoplus_{I_{n-1}} \Dk{N} * Q \arrow[d, hook] \arrow[r] & \cdots \\
            \cdots \arrow[r] & \bigoplus_{I_{n+1}} \Dk{G} \arrow[r] & \bigoplus_{I_n} \Dk{G} \arrow[r] & \bigoplus_{I_{n-1}} \Dk{G} \arrow[r] & \cdots \nospacepunct{,}
        \end{tikzcd}
    \]
    where the upper chain complex computes \(\H_\bullet(N;\Dk{N})\) and the lower chain complex computes \(\H_\bullet(G; \Dk{G})\).

    For any \(c \in Z_n(\bigoplus_{I_\bullet} \Dk{G})\) be an \(n\)-cycle, the Ore condition ensures there is some \(\alpha \in \Dk{N} * Q\) such that \(\alpha c \in \bigoplus_{I_\bullet} \Dk{N} * Q\). Then \((\Dk{N}*Q)\cdot \alpha c\) is an infinite-dimensional \(\Dk{N}\)-subspace of \(Z_n(\bigoplus_{I_\bullet} \Dk{N} * Q)\). But \(\btwo{n}(N;k) < \infty\), so there must be some \(d \in \bigoplus_{I_{n+1}} \Dk{N} * Q\) such that \(\partial d = \beta \cdot \alpha c\) for some nonzero \(\beta \in \Dk{N} * Q\). Hence, \(\partial ((\beta \alpha)\inv d) = c\), which proves that \(\btwo{n}(G;k) = 0\). \qedhere
\end{proof}





\section{Iterated fractions}

We fix some notations that will be used throughout the section.

\begin{notation}
    Let \(\ses{N}{G}{Q}\) be a fixed short exact sequence of groups, where \(G\) is locally indicable and \(\Dk{G}\) exists, and \(Q\) is torsion-free nilpotent with a fixed order \(\ol\). By \cref{lem:all_orders_are_lex}, there exists a central series
    \[
        1 = Q_n \leqslant \dots \leqslant Q_0 = Q
    \]
    such that every successive quotient \(Q_i/Q_{i+1}\) has an order \(\ol_i\), and \(\ol\) is the lexicographic order corresponding to the orders \(\ol_i\). More precisely, for nontrivial \(t \in Q\), this means that \(1 \prec t\) if and only if \(1 \prec_i \overline t\) in \(Q_i/Q_{i+1}\) where \(i\) is maximal such that \(t \in Q_i\).
\end{notation}

Let \(G_i \leqslant G\) be the preimage of \(Q_i\) in \(G\) under the quotient map. Consider the division rings defined by
\[
    \mathcal D_n = \Dk{N} \quad \text{and} \quad \mathcal D_i = \Ore(\mathcal D_{i+1} * Q_i/Q_{i+1}).
\]
By reverse induction, it is easy to see that \(\mathcal D_i \cong \Dk{G_i}\). Hence, \(\Dk{G} = \bigcup_{i=0}^n \mathcal D_i\).

\begin{defn}[Iterated fractions]\label{def:iterated_fraction}
    Let \(f \in \Dk{G}\) and let \(i\) be the maximal integer such that \(f \in \mathcal D_i\). If \(i = n\), then we (trivially) say that \(f\) is an \emph{iterated fraction} for itself. If \(i < n\), then an \emph{iterated fraction} \(\mathfrak f\) for \(f\) consists of the following data:
    \begin{enumerate}
        \item a pair of elements \(\alpha, \beta\) in \(\mathcal D_{i+1} * Q_i/Q_{i+1}\), where \(\beta \neq 0\) and \(f = \alpha \beta\inv\), and
        \item iterated fractions for every coefficient of \(\alpha\) and \(\beta\).
    \end{enumerate}
    In this situation, we say that \(\mathfrak f\) \emph{represents} \(f\) and, by abuse of notation, we often write \(\mathfrak f = (\alpha, \beta)\).
\end{defn}

\begin{defn}[Compatibility]
    Let \(f \in \Dk{G}\), let \(\mathfrak{f}\) be an iterated fraction for \(f\), and let \(\chi = (\chi_i)_{i=0}^{n-1}\) be a multicharacter on \(Q\) associated to the central series \(1 = Q_n \leqslant \dots \leqslant Q_0 = Q\). Let \(i\) be the maximal integer such that \(f \in \mathcal D_i\).
    \begin{enumerate}
        \item If \(i = n\), then \(\chi\) and \(\mathfrak f\) are (trivially) \emph{compatible}.
        \item If \(i < n\) and \(\mathfrak f = (\alpha, \beta)\), then \(\mathfrak f\) and \(\chi\) are \emph{compatible} if \(\chi_i \colon Q_i/Q_{i+1} \rightarrow \R\) strictly preserves the order on \(\supp(\alpha) \cup \supp(\beta)\) (where \(Q_i/Q_{i+1}\) has the order \(\ol_i\) and \(\R\) has its standard order) and \(\chi\) is compatible with the iterated fraction of every coefficient of \(\alpha\) and \(\beta\).
    \end{enumerate}
\end{defn}

The following lemma is a generalised version of \cite[Lemma 2.5]{OkunSchreve_DawidSimplified}.

\begin{lem}\label{lem:orderable_Abelian_maps_to_R}
    Let \(\Gamma\) be an ordered Abelian group and suppose that \(g_1 \prec \cdots \prec g_n\) is a finite increasing sequence of elements in \(\Gamma\). Then there is a homomorphism \(\varphi \colon \Gamma \rightarrow \R\) such that \(\varphi(x_1) < \dots < \varphi(x_n)\).
\end{lem}
\begin{proof}
    The \emph{divisible hull} of \(\Gamma\) is the ordered group \(\Gamma \otimes \Q\), where \(g \otimes \frac{1}{m} \prec h \otimes \frac{1}{n}\) if and only if \(ng \prec mh\). The map \(g \mapsto g \otimes 1\) is an order-preserving embedding \(\Gamma \hookrightarrow \Gamma \otimes \Q\) (in fact, \(\Gamma \otimes \Q\) is the unique minimal orderable divisible group that contains \(\Gamma\) as an ordered subgroup). Thus, it suffices to assume that \(\Gamma\) has a \(\Q\)-vector space structure.

    Let \(V \leqslant \Gamma\) be the \(\Q\)-subspace generated by \(x_1, \dots, x_n\), and let \(A \subseteq V\) be the group generated by \(x_1, \dots, x_n\). Then \(A\) admits a map to \(\R\) that strictly preserves the order on \(\{x_1, \dots, x_n\}\). This map uniquely extends to a homomorphism \(V \rightarrow \R\), which can then be extended to all of \(\Gamma\) by setting it to be zero on a complementary subspace of \(V\).  Alternatively, one could use Zorn's Lemma. \qedhere
\end{proof}

\begin{lem}\label{lem:compatible_character}
    Let \(f_1, \dots, f_l \in \Dk{G}\) be a finite collection of elements and let \(\mathfrak f_j\) be an iterated fraction for \(f_j\) for each \(j = 1, \dots, l\). Then there is a multicharacter on \(Q\) associated to the central series \(1 = Q_n \leqslant \dots \leqslant Q_0  = Q\) that is compatible with each \(\mathfrak f_j\).
\end{lem}
\begin{proof}
    Let \(i\) be maximal such that \(f_j \in \mathcal D_i\) for all \(j = 1, \dots, l\). If \(i = n\), then every multicharacter is trivially compatible with the iterated fractions \(\mathfrak f_j\). Now suppose that \(i < n\), let \(\mathfrak f_j = (\alpha_j, \beta_j)\), and let \(F_j = \{\mathfrak f_{j,k}\}\) be the set of iterated fractions representing the (nonzero) coefficients of \(\alpha_j\) and \(\beta_j\). By induction, there is a multicharacter \((\chi_{i+1}, \dots, \chi_n)\) that is compatible with every iterated fraction in \(\bigcup_j F_j\). Let \(\chi_i \colon Q_i/Q_{i+1} \rightarrow \R\) strictly preserve the order of \(\supp(\alpha) \cup \supp(\beta)\), which exists by \cref{lem:orderable_Abelian_maps_to_R}. For \(m < i\), we let \(\chi_m \colon Q_m/Q_{m+1} \rightarrow \R\) be any character. Then \(\chi = (\chi_i)_{i=0}^{n-1}\) is compatible with each iterated fraction \(\mathfrak f_i\). \qedhere
\end{proof}

\begin{prop}\label{prop:compatible_novikov}
    Let \(\mathfrak f\) be an iterated fraction for \(f \in \Dk{G} \smallsetminus \{0\}\) and let \(\chi\) be a multicharacter on \(Q\) associated to the central series \(1 = Q_n \leqslant \dots \leqslant Q_0 = Q\). If \(\chi\) and \(\mathfrak f\) are compatible, then \(f, f\inv \in \cnov{\Dk{N}*Q}\).
\end{prop}
\begin{proof}
    Let \(i\) be maximal such that \(f \in \mathcal D_i\). By induction on \(i\), we will show that
    \[
        f, f\inv \in \nov{\Dk{N}*Q_i}{(\chi_i, \dots, \chi_n)} =: R_i.
    \]
    If \(i = n\), then the result is trivial. Let \(\mathfrak f = (\alpha, \beta)\), where \(\alpha, \beta \in \mathcal D_{i-1} * Q_i/Q_{i-1}\). By induction, the coefficients (and their inverses) of \(\alpha\) and \(\beta\) all lie in \(R_{i+1}\). Write \(\beta = (1 - \beta_+) \beta_0\), where \(\beta_0 = r u_q\) for some \(r \in R_{i+1}^\times\) and \(q \in Q_i/Q_{i+1}\) and such that \(\chi_i\) takes positive values on \(\supp(\beta_+)\). This implies that 
    \[
        (1 - \beta_+)\inv = \sum_{m=0}^\infty \beta_+^m \ \in \ \nov{R_{i-1} * Q_i/Q_{i-1}}{\chi_i} = R_i,
    \]
    since \(\beta\), and therefore \(\beta_+\), has coefficients in \(R_{i+1}\). Therefore, 
    \[
        f = \alpha\beta\inv = \alpha \beta_0\inv \sum_{m=0}^\infty \beta_+^m \ \in \ R_i.
    \]
    The result then follows, since \(R_n \subseteq \dots \subseteq R_0 = \cnov{\Dk{N}*Q}\). \qedhere
\end{proof}

The following corollary is immediate; it gives a sufficient condition for elements of \(\Dk{G}\) to be represented in Novikov rings.

\begin{cor}\label{cor:compatible_novikov}
    If \(f \in \Dk{G} \cap (kN *_\prec Q)\) and \(\chi\) is compatible with \(\mathfrak f\), then \(f \in \cnov{kG}\).
\end{cor}
\begin{proof}
    By \cref{prop:compatible_novikov}, \(f \in \cnov{\Dk{N}*Q}\). But we assumed that \(f \in kN *_\prec Q\), which means that the coefficients in the series representation of \(f\) are in \(kN\) (see \cref{fig:big_diagram}). Hence, \(f \in \cnov{kN * Q} = \cnov{kG}\). \qedhere
\end{proof}

\section{RPVN groups}

We now establish some properties of the class of \emph{residually (poly-\(\Z\) and virtually nilpotent)} groups (hereafter, \emph{RPVN} groups) and their Hughes-free division rings. 

\subsection{Residual properties of RPVN groups}

A group \(G\) is RPVN if for every nontrivial \(g \in G\), there is a poly-\(\Z\) virtually nilpotent quotient \(G \rightarrow Q\) that does not send \(g\) to the identity. If \(G\) is countable, then this is equivalent to the existence of a normal residual chain \(G = N_0 \geqslant N_1 \geqslant \dots\) such that \(G/N_i\) is poly-\(\Z\) virtually nilpotent. This follows from the fact that the class of poly-\(\Z\) virtually nilpotent groups is closed under direct products and subgroups. More precisely, given an enumeration \(g_1, g_2, \dots\) of \(G \smallsetminus \{1\}\), let \(G \rightarrow Q_i\) be a poly-\(\Z\) virtually nilpotent quotient in which \(g_i\) is nontrivial. Then we may take \(N_i\) to be the kernel of the induced map \(G \rightarrow Q_1 \times \dots \times Q_i\).

We recall some classical facts from the theory of polycyclic groups. Let \(G\) be a polycyclic-by-finite group, and suppose that 
\[
    1 = G_n \leqslant G_{n-1} \leqslant \dots \leqslant G_0 = G
\]
is a normal series whose successive quotients \(G_i/G_{i+1}\) are either finite and nontrivial minimal normal subgroups of \(G/G_{i+1}\) or free Abelian and have no infinite index \(G\)-invariant subgroups (under the action of \(G\) induced by conjugation). Such a series is called a \emph{weak chief series} for \(G\), and every polycyclic-by-finite group has one. Recall that \(G\) contains a unique maximal nilpotent normal subgroup \(\Fitt(G)\) called the \emph{Fitting subgroup} of \(G\); i.e.~\(\Fitt(G)\) is nilpotent and it contains every nilpotent normal subgroup of \(G\).

\begin{prop}[{\cite[Lemma 4]{Humphreys_polycyclic}}]
    Let \(G\) be a polycyclic-by-finite group with a weak chief series \(1 = G_{n+1} \leqslant G_n \leqslant \dots \leqslant G_1 \leqslant G_0 = G\). Then
    \[
        1 \leqslant \Fitt(G) \cap G_n \leqslant \dots \leqslant \Fitt(G) \cap G_1 \leqslant \Fitt(G)
    \]
    is a central series for \(\Fitt(G)\).
\end{prop}

We now establish a residual property of RPVN groups that will be crucial in what follows. We will say that a residual chain \((N_i)_{i\geqslant 0}\) \emph{refines} another residual chain \((M_j)_{j \geqslant 0}\) in the same group \(G\) if for every \(j \geqslant 0\) there is an index \(i_j\) such that \(M_j = N_{i_j}\).

\begin{prop}\label{prop:RPVN}
    Let \(G\) be a countable RPVN group and suppose that \((K_l)_{l\geqslant 0}\) is a normal residual chain for \(G\) such that \(G/K_l\) is poly-\(\Z\) virtually nilpotent for each \(l\). Then there is a normal residual chain \((N_i)_{i\geqslant 0}\) for \(G\) that refines \((K_l)_{l\geqslant0}\) and a normal residual chain \((G_i)_{i\geqslant0}\) of finite-index subgroups such that 
    \begin{enumerate}[label = (\arabic*)]
        \item\label{item:chainInclusion} \(N_i \leqslant G_i\) for all \(i \geqslant 0\);
        \item\label{item:succQuotients} \(N_i/N_{i+1}\) is free Abelian and \(G/N_i\) is poly-\(\Z\) and virtually nilpotent for each \(i \geqslant 0\);
        \item\label{item:centralSeries} \(G_i/N_i\) is torsion-free nilpotent and
        \[
            1 \leqslant (G_i \cap N_{i-1})/N_i \leqslant \dots \leqslant (G_i \cap N_1)/N_i \leqslant G_i/N_i
        \]
        is a central series for all \(i \geqslant 0\).
    \end{enumerate}
\end{prop}

\begin{defn}\label{def:interlaced_chains}
    If \((G_i)\) and \((N_i)\) are normal residual chains in an RPVN group \(G\) satisfying the conclusions \ref{item:chainInclusion}, \ref{item:succQuotients}, and \ref{item:centralSeries} of \cref{prop:RPVN}, then we say that \((G_i)\) and \((N_i)\) are \emph{interlaced}.
\end{defn}

\begin{proof}[Proof (of \cref{prop:RPVN})]
    For every \(l \geqslant 0\), \(K_l/K_{l+1}\) is poly-\(\Z\) virtually nilpotent. We construct a finite characteristic series with free Abelian successive quotients as follows: the kernel of the free Abelianisation map
    \[
        K_l/K_{l+1} \rightarrow (K_l/K_{l+1})\ab \otimes \Q
    \]
    is characteristic and of lower cohomological dimension than that of \(K_l/K_{l+1}\). By induction, this kernel has a characteristic series with free Abelian successive quotients, and therefore so does \(K_l/K_{l+1}\). For each \(l\), we take the preimage of the characteristic series of \(K_l/K_{l+1}\) in \(G\) and obtain a new normal residual chain \((L_j)\) which refines \((K_l)\). Note that the successive factors of \((L_j)\) are all free Abelian.

    The conjugation action of \(G\) on itself induces an action on each Abelian group \(L_j/L_{j+1}\). If \(G\) preserves an infinite-index subgroup of \(L_j/L_{j+1}\), then it preserves the minimal direct factor containing it. Thus, we may further refine the residual chain \((L_j)\) to a normal residual chain \((N_i)\) such that
    \[
        1 \leqslant N_{i-1}/N_i \leqslant \dots \leqslant N_1/N_i \leqslant N_0/N_i = G/N_i
    \]
    is a weak chief series for \(G/N_i\). Hence, \((N_i)\) satisfies the conclusion of item \ref{item:succQuotients}.

    We now construct the chain \((G_i)\) by induction. Let \(G_0 = N_0 = G\). Fix enumerations
    \[
        \{g_1^{(i)}, g_2^{(i)}, g_3^{(i)}, \dots\} = N_{i-1} \smallsetminus N_i 
    \]
    for all \(i \geqslant 1\). Suppose we have constructed a chain of finite-index normal subgroups \(G_0 \geqslant \dots \geqslant G_{i-1}\) such that \(N_i \leqslant G_i\) and \(G_i/N_i\) is torsion-free nilpotent, and moreover such that \(g_1^{(m)}, g_2^{(m-1)}, \dots, g_m^{(1)} \notin G_m\) for all \(m < i\). Since \(G/N_i\) is polycyclic, it is residually finite. Thus, we let \(F_i\) be a finite quotient of \(G/N_i\) in which \(g_1^{(i)}, g_2^{(i-1)}, \dots, g_i^{(1)}\) are all nontrivial, and let 
    \[
        G_i = \ker(G \rightarrow F_i) \cap G_{i-1} \cap H_i,
    \]
    where \(H_i\) is the preimage of \(\Fitt(G/N_i)\) in \(G\) (note that \([G:H_i]< \infty\), since \(G/N_i\) is virtually nilpotent). The \(g_1^{(i)}, g_2^{(i-1)} \dots, g_i^{(1)} \notin G_i\). By induction, \((G_i)\) is a residual chain of finite-index subgroups such that \(N_i \leqslant G_i\) and \(G_i/N_i\) is torsion-free nilpotent for all \(i\). 

    To conclude, since
    \[
        1 \leqslant \Fitt(G/N_i) \cap N_{i-1}/N_i \leqslant \dots \leqslant \Fitt(G/N_i) \cap N_1/N_i \leqslant \Fitt(G/N_i)
    \]
    is a central series for \(\Fitt(G/N_i)\) and \(G_i/N_i \leqslant \Fitt(G/N_i)\), it follows that
    \[
        1 \leqslant (G_i \cap N_{i-1})/N_i \leqslant \dots \leqslant (G_i \cap N_1)/N_i \leqslant G_i/N_i
    \]
    is a central series for \(G_i/N_i\). \qedhere
\end{proof}

\subsection{The division ring of an RPVN group}

Note that if \(G\) is RPVN, then \(G\) is residually (locally indicable and amenable) and therefore there exists a Hughes-free embedding \(kG \hookrightarrow \Dk{G}\) for every field \(k\) by \cite[Corollary 1.3]{JaikinZapirain2020THEUO}. We now turn to the description of the Hughes-free division ring of an RPVN group  in terms of Novikov rings of its finite-index subgroups, similar to the one given by Kielak \cite{KielakRFRS} and subsequently by Okun--Schreve \cite{OkunSchreve_DawidSimplified} in the case of RFRS groups. We start by recalling the definitions of invariant Malcev--Neumann rings and inductive rings introduced by Okun--Schreve in \cite{OkunSchreve_DawidSimplified}.

\begin{defn}[Invariant Malcev--Neumann rings]
    Let \(\mathcal G\) be a locally indicable group and let \(k\) be a field such that \(\Dk{\mathcal G}\) exists. Suppose that \(N \trianglelefteqslant G\) is a pair of normal subgroups of \(\mathcal G\) such that \(G/N\) is orderable and let \(R \subseteq \Dk{N}\) be a ring containing \(kN\) that is invariant under the conjugation action of \(\mathcal G\) on \(\Dk{N}\). The \emph{invariant Malcev--Neumann ring} is
    \[
        R * \langle G/N \rangle = \left\{ f \in \Dk{G} \ : \ \iota_\olob(f) \in R *_\olob G/N \ \text{for all orders} \ \olob \ \text{on} \ G/N \right\}.
    \]
\end{defn}

\begin{defn}[Inductive rings]\label{def:inductive}
    Let \(G\) be a group with a normal residual chain \((N_i)_{i\geqslant 0}\) of normal subgroups such that \(N_i/N_{i+1}\) is orderable and amenable for all \(i \geqslant 0\). Then we define
    \[
        R_0^{(N_i)} = kN_0, \quad R_1^{(N_i)} = k N_1 * \langle N_0/N_1 \rangle, \quad R_2^{(N_i)} = kN_2 * \langle N_1/N_2 \rangle* \langle N_0/N_1 \rangle,
    \]
    and in general
    \[
        R_m^{(N_i)} = kN_m * \langle N_{m-1}/N_m \rangle * \langle N_{m-2}/N_{m-1} \rangle * \cdots * \langle N_0/N_1 \rangle.
    \]
\end{defn}

\begin{rem}
    We have chosen not to bracket the expressions for the rings \(R_n\) above for reasons of readability and conciseness, and because there is only one possible interpretation of the notation that makes sense.
\end{rem}

\begin{thm}[{\cite[Theorem 5.1]{OkunSchreve_DawidSimplified}}] \label{thm:OkunSchreve}
    Let \(G\) be a group with a normal residual chain \((N_i)_{i\geqslant }\) such that \(N_i/N_{i+1}\) is orderable and amenable for all \(i\). Then
    \[
        \Dk{G} = \bigcup_{m=0}^\infty R_m^{(N_i)}.
    \]
\end{thm}

\begin{rem}
    The group \(G\) in the above theorem is residually (amenable and locally indicable), so \(\Dk{G}\) exists by \cite[Corollary 1.3]{JaikinZapirain2020THEUO}.
\end{rem}

If \(H \leqslant G\) is a subgroup and \(G\) has a normal residual chain \((N_i)_{i\geqslant 0}\) with orderable and amenable successive quotients, then \((H\cap N_i)_{i\geqslant 0}\) is a normal residual chain for \(H\) with orderable and amenable successive quotients. The inductive rings of finite-index subgroups of \(G\) are related to the inductive rings of \(G\) in the following natural way, provided that the successive quotients of the residual chain are torsion-free Abelian. Note that the following lemma applies to RPVN groups by \cref{prop:RPVN}.

\begin{lem}[{\cite[Theorem 5.2]{OkunSchreve_DawidSimplified}}] \label{lem:OSfinite_index}
    Let \(G\) be a residually (locally indicable and amenable group) such that there is a normal residual chain \((N_i)_{i \geqslant 0}\) with torsion-free Abelian successive quotients \(N_i/N_{i+1}\). If \(H \trianglelefteqslant G\) is a normal subgroup of finite index, then
    \[
        R_m^{(N_i)} \cong R_m^{(H \cap N_i)} * G/H
    \]
    for each \(i \geqslant 0\).
\end{lem}

We are now ready to give the description of \(\Dk{G}\) in terms of Novikov rings. Given the preparatory results above, the proof is similar to that of \cite[Theorem 7.1]{OkunSchreve_DawidSimplified}.

\begin{conv}\label{conv:pm}
    In what follows, if \(\chi = (\chi_1, \dots, \chi_n)\) is a multicharacter, then \(\nov{kG}{\pm \chi}\) will stand for all the \(2^n\) Novikov rings \(\nov{kG}{(\pm\chi_1, \dots, \pm\chi_n)}\).
\end{conv}

\begin{thm}\label{thm:divring_elements}
    Let \(G\) be a countable RPVN group with interlaced residual chains \((G_i)\) and \((N_i)\), and let \(f_1, \dots, f_l \in \Dk{G}\). Then there is some \(m \in \N\) and a multicharacter on \(G_m/N_m\) such that, when writing
    \begin{equation}\label{eq:coeffs}\tag{\(\dagger\)}
        f_j = \sum_{t \in G/G_m} f_j^t t \in \Dk{G_m} * G/G_m, \qquad f_i^t \in \Dk{G_m},
    \end{equation}
    we have \(f_j^t \in \widehat{kG_m}^{\pm \chi}\) for all \(j = 1, \dots, l\) and \(t \in G/G_m\).
\end{thm}
\begin{proof}
    By \cref{thm:OkunSchreve}, there is some \(m \in \N\) such that \(f_j \in R_m^{(N_i)}\) for all \(j = 1, \dots, l\). Let \(H = G_m\). By \cref{prop:RPVN}, the series
    \[
        1 \leqslant (H \cap N_{m-1})/N_m \leqslant \dots \leqslant (H \cap N_1)/N_m \leqslant H/N_m
    \]
    is a central series for \(H/N_m\). Put orders \(\olob_i\) on the successive quotients
    \[
        (H \cap N_{i})/N_m \ / \ (H \cap N_{i+1})/N_m \cong (H \cap N_{i}) / (H \cap N_{i+1})
    \]
    and let \(\olob\) be the induced lexicographic order on \(H/N_m\). We have
    \begin{align*}
        R_m^{(H \cap N_i)} &= k N_m * \langle (H \cap N_{m-1})/N_m \rangle * \cdots * \langle H / (H \cap N_1) \rangle \\
        &\leqslant k N_m *_{\olob_{m-1}} (H \cap N_{m-1})/N_m *_{\olob_{m-2}} \cdots *_{\olob_0} H / (H \cap N_1) \\
        &= k N_m *_\olob H/N_m.
    \end{align*}
 
    By \cref{lem:OSfinite_index},
    \[
        f_j = \sum_{t \in G/H} f_j^t t \in R_m^{(N_i)} = R_m^{(H \cap N_i)} * G/H 
    \]
    and therefore \(f_j^t \in R_m^{(H \cap N_i)}\) for all \(j = 1, \dots, l\) and \(t \in G/H\). Fix representative fractions \(\mathfrak f_j^t\) for each of the elements \(f_j^t\) and let \(\chi\) be a compatible multicharacter on \(H/N_m\), which exists by \cref{lem:compatible_character}. It follows from \cref{cor:compatible_novikov} that \(f_j^t \in \widehat{kH}^\chi\). By reversing any of the orders \(\olob_i\) and changing the sign of the corresponding components of \(\chi\), we find that \(f_j^t \in \widehat{kH}^{\pm\chi}\). \qedhere
\end{proof}

The following theorem is the cohomological generalisation of \cite[Theorem 5.2]{KielakRFRS} to RPVN groups. In particular, it implies that vanishing \(\ell^2\)-cohomology implies vanishing Novikov (co)homology. The proof is similar to that given in \cite[Proposition 3.2]{Fisher_freebyZ}.

\begin{thm}\label{thm:L2_novikov}
    Let \(G\) be a countable RPVN group, let \(k\) be a field, and let \((G_i)\) and \((N_i)\) be interlaced chains in \(G\). Suppose \(P_\bullet\) is a chain complex of projective modules such that \(P_n\) is finitely generated for some \(n \in \N\). If \(\H^n(P_\bullet; \Dk{G}) = 0\), then there is a multicharacter \(\chi\) on \(G_m/N_m\) for some \(m\) such that
    \[
        \H^n(P_\bullet; \nov{kG_m}{\pm\chi}) = 0.
    \]
\end{thm}
\begin{proof}
    We may assume that \(P_\bullet\) is a chain complex of free \(kG\)-modules such that \(P_n\) is finitely generated. Indeed, let \(Q_n\) be a projective module such that \(P_n \oplus Q_n\) is finitely generated and free. The cohomology of the chain complex
    \[
        \cdots \rightarrow P_{n+2} \rightarrow P_{n+1} \oplus Q_{n+1} \rightarrow P_n \oplus Q_n \rightarrow P_{n-1} \rightarrow \cdots
    \]
    (with any coefficients) is isomorphic to that of \(P_\bullet\). By iteratively taking free complements of the other projective modules in the chain complex, we may assume that they are all free, and that \(P_n\) is finitely generated.

    We fix some notations and terminology that will be used in the rest of the proof. If \(M\) is a \(kG\)-module, then we denote \(\Hom_{kG}(M,\Dk{G})\) by \(M^*\). For each free \(kG\)-module \(P_i\), fix an isomorphism \(P_i \cong \bigoplus_{J_i} kG\). We say that a submodule \(F_i \leqslant P_i\) is a \emph{direct summand} if it corresponds to \(\bigoplus_{J_i'} kG\) for some \(J_i' \subseteq J_i\). Note that this is more restrictive than simply being a factor in a direct factor in an abstract direct sum decomposition of \(P_i\).

    We now focus on the portion \(P_{n+1} \rightarrow P_n \rightarrow P_{n-1}\) of the chain complex. We want to choose finitely generated direct summands \(F_{n\pm1} \leqslant P_{n\pm1}\) so that the cohomology of \(F_{n+1}^* \leftarrow P_n^* \leftarrow F_{n-1}^*\) coincides with that of \(P_{n+1}^* \leftarrow P_n^* \leftarrow P_{n-1}^*\). Since \(P_n\) is finitely generated, its image in \(P_{n+1}\) lies in a finitely generated direct summand \(F_{n-1} \leqslant P_{n+1}\). Then the images of the maps \(P_{n-1}^* \rightarrow P_n^*\) and \(F_{n-1}^* \rightarrow P_n^*\) coincide, since every map on \(F_{n-1}\) extends to one on \(P_{n-1}\). Note that this did not have anything to do with the choice of coefficients in \(\Dk{G}\).

    Since \(P_{n+1}\) is the directed union of its finitely generated direct summands, it follows that 
    \[
        \ker(P_n^* \rightarrow P_{n+1}^*) = \bigcap_{\substack{F_{n+1} \leqslant P_{n+1} \\ \text{f.g.~direct summand}}} \ker(P_n^* \rightarrow F_{n+1}^*).
    \]
    Because \(\Dk{G}\) is a division ring, each \(\ker(P_n^* \rightarrow F_{n+1}^*)\) is a finite-dimensional \(\Dk{G}\)-module and every chain of such modules must have a minimal element (under inclusion). By Zorn's Lemma, there is a minimal space, and therefore 
    \[
        \ker(P_n^* \rightarrow P_{n+1}^*) = \ker(P_n^* \rightarrow F_{n+1}^*)
    \]
    for some finitely generated direct summand \(F_{n+1} \leqslant P_{n+1}\). Let \(F_n = P_n\). Then the degree \(n\) cohomology of of \(F_{n+1}^* \leftarrow F_n^* \leftarrow F_{n-1}^*\) coincides with that of \(P_{n+1}^* \leftarrow P_n^* \leftarrow P_{n-1}^*\), as desired.

    The coboundary maps of \(F_{n+1}^* \leftarrow F_n^* \leftarrow F_{n-1}^*\) are maps between finitely generated modules over \(\Dk{G}\), which we identify with finite matrices (note that these matrices will have entries in \(kG\), since they are induced by maps between free \(kG\)-modules). Since \(\Dk{G}\) is a division ring, for \(i = n-1, n, n+1\) there are invertible matrices
    \[
        M_i \colon \Hom_{kG}(F_i, \Dk{G}) \rightarrow \Hom_{kG}(F_i, \Dk{G})
    \]
    that put the coboundary maps into Smith normal form. By \cref{thm:divring_elements}, there is some \(H = G_m\) such that every entry of each matrix \(M_i^{\pm 1}\) has coefficients in \(\nov{kH}{\pm\chi}\) when written in the form \eqref{eq:coeffs}. Then there are invertible matrices
    \[
        \overline M_i^{\pm 1} \colon \Hom_{kH}(F_i, \Dk{G}) \rightarrow \Hom_{kH}(F_i, \Dk{G})
    \]
    all of whose entries are in \(\nov{kH}{\pm\chi}\). Moreover, they put the cochain complex 
    \[\label{eq:modified_nov_complex}
        \Hom_{kH}(F_{n+1}, \nov{kH}{\pm\chi}) \leftarrow \Hom_{kH}(F_n, \nov{kH}{\pm\chi}) \leftarrow \Hom_{kH}(F_{n-1}, \nov{kH}{\pm\chi}) \tag{$\ddagger$}
    \]
    into \([G:H]\) copies of the same Smith normal form as do the matrices \(M_i^{\pm 1}\), and therefore its degree \(n\) cohomology vanishes.
    
    To conclude, note that \(\Hom_{kH}(F_{i-1}, \nov{kH}{\pm\chi})\) and \(\Hom_{kH}(P_{i-1}, \nov{kH}{\pm\chi})\) have the same image in \(\Hom_{kH}(P_i, \nov{kH}{\pm\chi})\). Since the kernel of 
    \[
        \Hom_{kH}(P_i, \nov{kH}{\pm\chi}) \rightarrow \Hom_{kH}(P_{i+1}, \nov{kH}{\pm\chi})
    \]
    is contained in that of 
    \[
        \Hom_{kH}(P_i, \nov{kH}{\pm\chi}) \rightarrow \Hom_{kH}(F_{i+1}, \nov{kH}{\pm\chi}),
    \]
    and thus we conclude that \(\H^n(H;\nov{kH}{\pm\chi})\) is a submodule of the degree \(n\) cohomology of \eqref{eq:modified_nov_complex}, and therefore it vanishes. \qedhere
\end{proof}

\section{Cohomological dimension of conilpotent subgroups}

In this section, we prove the main results stated in the introduction. In \cite[Theorem 3.5]{Fisher_freebyZ}, it is shown that if \(\chi \colon G \rightarrow \R\) is a character and \(G\) is a group of type \(\FP(R)\) for some ring \(R\), then the vanishing of the top-dimensional Novikov cohomology with respect to \(\pm\chi\) implies that \(\cd_R(\ker \chi) < \cd_R(G)\). We extend this result to multicharacters on nilpotent quotients. We first prove a lemma, which abstracts the idea of the proof of \cite[Theorem 3.5]{Fisher_freebyZ}. 

\begin{lem}\label{lem:very_general}
    Let \(R\) be a ring and let \(1 \rightarrow N \rightarrow G \rightarrow Q \rightarrow 1\) be a short exact sequence of groups such that \(G\) is of type \(\FTP_n(R)\), and let \(M\) be an \(RN\)-module. If there exist collections of \(RG\)-rings \(\{S_i\}_{i \in I}\), \(S_i\)-modules \(\{L_i\}_{i\in I}\), and homomorphisms \(\{M_i \rightarrow M * \llbracket Q \rrbracket\}_{i\in I}\) of \(RG\)-modules such that the induced map
    \[
        \bigoplus_{i \in I} L_i \rightarrow M * \llbracket Q \rrbracket
    \]
    is surjective and \(\H^n(G; S_i) = 0\) for all \(i \in I\), then \(\H^n(N;M) = 0\).
\end{lem}
\begin{proof}
    Fix a projective resolution \(0 \rightarrow P_n \rightarrow \dots \rightarrow P_0 \rightarrow R \rightarrow 0\) of the trivial \(RG\)-module \(R\) such that \(P_n\) is finitely generated. We will prove that \(\H^(N,M) = 0\) by a series of reductions.

    \begin{claim}
        It is suffices to show that \(\H^n(G; M* \llbracket Q \rrbracket) = 0\).
    \end{claim}
    \begin{proof}
        This is because \(M * \llbracket Q \rrbracket\) is isomorphic to the \(RG\)-module coinduced by \(M\), and therefore \(\H^n(G; M* \llbracket Q \rrbracket) \cong \H^n(N;M)\) by Shapiro's Lemma. \qedhere
    \end{proof}

    \begin{claim}
        It is suffices to show that \(\H^n(G; \bigoplus_{i \in I}L_i) = 0\).
    \end{claim}
    \begin{proof}
        Let \(K\) be the kernel of the surjection \(\bigoplus_{i \in I} L_i \rightarrow M * \llbracket Q \rrbracket\). Then the long exact sequence in cohomology associated to the surjection contains the portion
        \[
            \H^n(G; \bigoplus_{i \in I}L_i) \rightarrow \H^n(G; M* \llbracket Q \rrbracket) \rightarrow \H^{n+1}(G; K),
        \]
        where \(K\) is the kernel of the surjection. But \(\cd_R(G) \leqslant n\), so \(\H^n(G; \bigoplus_{i \in I}L_i)\) surjects onto \(\H^n(G; M* \llbracket Q \rrbracket)\), which proves the claim. \qedhere
    \end{proof}

    \begin{claim}
        It suffices to show that \(\H^n(G;L_i) = 0\) for all \(i \in I\).
    \end{claim}
    \begin{proof}
        Because \(P_n\) is finitely generated, there is a natural isomorphism
        \[
            \bigoplus_{i \in I} \Hom_{RG}(P_n,L_i) \cong \Hom_{RG}(P_n, \bigoplus_{i\in I}L_i).
        \]
        Thus, there is a commutative diagram
        \[
            \begin{tikzcd}
                \bigoplus_{i\in I}\Hom_{RG}(P_n;L_i) \arrow[d, "\cong"'] \arrow[r] & \bigoplus_{i\in I}\H^n(G;L_i) \arrow[d] \arrow[r] & 0 \\
                \Hom_{RG}(P_n;\bigoplus_{i\in I}L_i) \arrow[r] & \H^n(G;\bigoplus_{i\in I}L_i) \arrow[r] & 0
            \end{tikzcd}
        \]
        with exact rows. This immediately implies that the rightmost vertical map is an epimorphism, and thus proves the claim. \qedhere
    \end{proof}

    \begin{claim}
        The natural map \(\H^n(G;S_i) \otimes_{S_i} L_i \rightarrow \H^n(G;L_i)\) is surjective for all \(i \in I\).
    \end{claim}
    \begin{proof}
        Because \(P_n\) is finitely generated, there are natural isomorphisms
        \begin{align*}
            \Hom_{RG}(P_n, L_i) &\cong \Hom_{RG}(P_n,RG) \otimes_{RG} L_i \\
            &\cong \Hom_{RG}(P_n,RG) \otimes_{RG} S_i \otimes_{S_i} L_i \\
            &\cong \Hom_{RG}(P_n,S_i) \otimes_{S_i} L_i.
        \end{align*}
        and because tensoring is right exact, the commutative diagram
        \[
            \begin{tikzcd}
                \Hom_{RG}(P_n,S_i) \otimes_{S_i} L_i \arrow[d, "\cong"'] \arrow[r] & \H^n(G;S_i) \otimes_{S_i} L_i \arrow[d] \arrow[r] & 0 \\
                \Hom_{RG}(P_n;L_i) \arrow[r] & \H^n(G;L_i) \arrow[r] & 0
            \end{tikzcd}
        \]
        has exact rows. Therefore the rightmost vertical map is surjective, which proves the claim. \qedhere
    \end{proof}
    The conclusion \(\H^n(N,M) = 0\) then follows immediately from the claims, since we assume \(\H^n(R;S_i) = 0\) for all \(i \in I\). \qedhere
\end{proof}

Using \cref{lem:very_general}, we can relate the vanishing of top-dimensional Novikov cohomology with the cohomological dimension of a conilpotent kernel. Note that \cref{conv:pm} is still in place.

\begin{thm}\label{thm:cd_kernel}
    Let \(R\) be a ring and let \(\varphi \colon G \rightarrow Q\) be an epimorphism onto a nilpotent group \(Q\), where \(G\) is a group of type \(\FTP_d(R)\). If there is a multicharacter \(\chi = (\chi_i)_{i=0}^{n-1}\) on \(Q\) such that 
    \[
        \H^d(G; \nov{kG}{\pm \chi}) = 0,
    \]
    then \(\cd_k(\ker \varphi) < d\).
\end{thm}
\begin{proof}
    Let \(N = \ker \varphi\) and let \(M\) be any \(RN\)-module. Denote by 
    \[
        1 = Q_n \leqslant \dots \leqslant Q_0 = Q
    \]
    the central series associated to \(\chi\). We will show that \(\H^d(N;M) = 0\), which implies that \(\cd_R(N) < n\). There is a decomposition
    \[
        M*\llbracket Q \rrbracket \cong M * \llbracket Q_1/Q_0 \rrbracket * \dots * \llbracket Q_n/ Q_{n-1} \rrbracket,
    \]
    so we think of the elements of \(M*\llbracket Q \rrbracket\) as series of the form
    \[
        \sum_{q_n \in Q_n/Q_{n-1}} \dots \sum_{q_1 \in Q_1/Q_0} m_{q_1, \dots, q_n} u_{q_1} \cdots u_{q_n}.
    \]
    Given an \(n\)-tuple \((t_1, \dots, t_n) \in \R^n\), let
    \[
        L_{t_1, \dots, t_n}^{\pm, \dots, \pm} \ = \ \left\{ \sum m_{q_1, \dots, q_n} u_{q_1} \cdots u_{q_n}  \ : \ \pm \chi_i(q_i) > \pm t_i \ \text{for each} \ i = 1, \dots, n \right\}
    \]
    and let \(L^{\pm, \dots, \pm} = \bigcup_{(t_1, \dots, t_n) \in \R^n} = L_{t_1, \dots, t_n}^{\pm, \dots, \pm}\). The inclusions \(L^{\pm,\dots, \pm} \rightarrow M*\llbracket Q \rrbracket\) induce a surjection \(\bigoplus_{\pm, \dots, \pm} L^{\pm, \dots, \pm} \rightarrow M * \llbracket Q \rrbracket\) and \(L^{\pm, \dots, \pm}\) is a \(\nov{kG}{(\pm\chi_1, \dots, \pm \chi_n)}\)-module. By \cref{lem:very_general}, \(\H^d(N;M) = 0\).
\end{proof}

We are now ready to prove \cref{thm:A}.

\begin{thm}\label{thm:topL2_cd_drop}
    Let \(G\) be a countable RPVN group of type \(\FTP_d(k)\) for some field \(k\). The following are equivalent:
    \begin{enumerate}
        \item \(\btwo{d}(G;k) = 0\);
        \item for every normal residual chain \((N_i)_{i \geqslant 0}\) such that \(G/N_i\) is poly-\(\Z\) virtually nilpotent, \(\cd_k(N_m) < d\) for sufficiently large \(m\);
        \item for every normal residual chain \((N_i)_{i \geqslant 0}\) such that \(G/N_i\) is poly-\(\Z\) virtually nilpotent, \(\hd_k(N_m) < d\) for sufficiently large \(m\).
    \end{enumerate}
\end{thm}
\begin{proof}
    By \cref{prop:RPVN}, it is enough to prove the claim in the case where there is a normal residual chain \((G_i)\) such that \((G_i)\) and \((N_i)\) are interlaced. Then \cref{thm:L2_novikov} implies there is \(m \in \N\) and a multicharacter \(\chi\) on \(G_m/N_m\) such that 
    \[
        \H^d(G_m; \nov{kG_m}{\pm \chi}) = 0.
    \]
    But then \(\cd_k(N_m) < d\) by \cref{thm:cd_kernel}.
    
    If \(\cd_k(N_m) < d\), then \(\hd_k(N_m) < d\) since the homological dimension of a group is always bounded above by its cohomological dimension. Finally, if \(\hd_k(G) < d\), then \(\btwo{d}(N_m;k) = 0 < \infty\). Since \(G/N_m\) is torsion-free and elementary amenable, \(\btwo{d}(G;k) = 0\) by \cref{prop:divring_gaboriau}. \qedhere
\end{proof}

We now apply our result to the class of residually torsion-free nilpotent (RTFN) groups and prove \cref{cor:B}, and to two-dimensional RPVN groups and prove \cref{cor:C}.

\begin{cor}\label{cor:RTFN}
    Let \(G\) be a finitely generated RTFN group of type \(\FTP_d(k)\) for some field \(k\). The following are equivalent:
    \begin{enumerate}
        \item \(\btwo{d}(G;k) = 0\);
        \item \(\cd_k(\gamma_i(G)) < d\) for \(i\) sufficiently large;
        \item \(\cd_k(\overline\gamma_j(G)) < d\) for \(j\) sufficiently large.
    \end{enumerate}
\end{cor}
\begin{proof}
    That (1) implies (3) follows immediately from \cref{thm:topL2_cd_drop}, and that (3) implies (2) follows immediately from the fact that \(\gamma_i(G) \leqslant \overline\gamma_i(G)\) for all \(i\). If (2) holds, then \(G\) has a nilpotent quotient with kernel \(\gamma_i(G)\) of cohomological dimension at most \(d-1\) over \(k\) for some \(i\), and therefore \(\btwo{d}(G;k) = 0\). Then \(\btwo{d}(G;k) = 0\) by \cref{prop:divring_gaboriau}. \qedhere
\end{proof}

Since group rings are left coherent if and only if they are right coherent, we will simply say that \(kG\) is \emph{coherent} if and only if it is left or right coherent.

\begin{cor}\label{cor:2dim_coherence}
    Let \(G\) be a finitely generated RPVN group with \(\cd_k(G) \leqslant 2\) for some field \(k\). The following are equivalent:
    \begin{enumerate}[label=(\arabic*)]
        \item\label{item:btwo2} \(\btwo{2}(G;k) = 0\);
        \item\label{item:free_by_LIVN} \(G\) is free-by-(poly-\(\Z\) and virtually nilpotent).
    \end{enumerate}
    In particular, if \(G\) is any RPVN group and \(\btwo{2}(G;k) = 0\), then \(G\) is a coherent group and \(kG\) is a coherent ring.
\end{cor}

\begin{rem}
    If \(G\) is RTFN, then we can conclude that \(G\) is free-by-(torsion-free nilpotent) in \ref{item:free_by_LIVN}.
\end{rem}

We thank Marco Linton for showing us the argument that a free-by-polycyclic group of cohomological dimension \(2\) is coherent.

\begin{proof}[Proof (of \cref{cor:2dim_coherence})]
    Suppose \ref{item:btwo2} holds. Then \(G\) is of type \(\FP_2(k)\) by \cite[Theorem 3.7]{JaikinLinton_coherence}, and therefore it is of type \(\FTP_2(k)\). Thus, \ref{item:free_by_LIVN} follows from \cref{thm:topL2_cd_drop} applied at \(d = 2\) and Swan's theorem, which states that torsion-free groups of cohomological dimension one over any field are free \cite{Swan_cd1}. If \ref{item:free_by_LIVN}, then \(\btwo{2}(G;k) = 0\) by \cref{prop:divring_gaboriau}.

    The equivalence of Items \ref{item:btwo2} and \ref{item:free_by_LIVN} follows from \cref{thm:topL2_cd_drop} applied at \(d = 2\), and Swan's Theorem, which states that torsion-free groups of cohomological dimension \(1\) over any field are free \cite{Swan_cd1}.

    We now turn to the statements about coherence. It suffices to assume that \(G\) is finitely generated, since \(\btwo{2}(H;k) = 0\) for all subgroups of \(H \leqslant G\) by \cite[Lemma 3.21]{FisherMorales_HNC}, and a group \(G\) (resp.~group algebra \(kG\)) is coherent if and only if \(H\) (resp.~\(kH\)) is coherent for all finitely generated subgroups \(H \leqslant G\). Hence, we can assume that \(G\) is free-by-(poly-\(\Z\)). It is then almost immediate that \(\Dk{G}\) is of weak dimension at most one as a \(kG\)-module (see \cite[Lemma 3.4]{Fisher_freebyZ}), and therefore \(kG\) is coherent by \cite[Corollary 3.2]{JaikinLinton_coherence}.
    
    Finally, we show that \(G\) is a coherent group. As \(G\) is free-by-(poly-\(\Z\)), we prove coherence of \(G\) by induction on the Hirsch length of the poly-\(\Z\) quotient. If the Hirsch length is \(0\), then \(G\) is free and therefore coherent. Now suppose that the Hirsch length is \(n > 0\). Then \(G\) splits as an ascending HNN extension of a group \(H\), which is free-by-(poly-\(\Z\) of Hirsch length \(n-1\)). By induction, \(H\) is coherent. By \cite[Theorem 1.3]{JaikinLinton_coherence}, it suffices to show that \(G\) is homologically coherent, i.e.~that every finitely generated subgroup of \(G\) is of type \(\FP_2(k)\). But this follows from the fact that \(kG\) is coherent. \qedhere
\end{proof}

Before giving some applications to homology growth and towards the Parafree Conjecture, we prove a lemma, which relates \(b_2(G)\) and \(\btwo{2}(G)\) for finitely generated groups of cohomological dimension \(2\). Had we assumed that \(G\) were finitely presented, or even of type \(\FP_2(\Q)\), the lemma would be an immediate consequence of \cite[Corollary 1.6]{JaikinZapirain2020THEUO}. We use Cohn's theory of specialisations between division rings to weaken the assumption (see \cref{subsec:Rdivrings}). Our proof is inspired by that of \cite[Proposition 2.2]{JaikinMorales_surfaceProfRigid}.

\begin{lem}\label{lem:removeFP2}
    Let \(G\) be a finitely generated group, let \(k\) be a field, and let \(\mathcal D_1\) and \(\mathcal D_2\) be \(kG\)-division rings such that there is a specialisation \(\rho \colon \mathcal D_1 \rightarrow \mathcal D_2\). If \(\cd_k(G) \leqslant 2\), then \(b_2(G; \mathcal D_1) \leqslant b_2(G; \mathcal D_2)\).
\end{lem}
\begin{proof}
    There is nothing to show if \(b_2(G; \mathcal D_2) = \infty\), so we assume \(b_2(G;\mathcal D_2) < \infty\). Let \(0 \rightarrow P_2 \rightarrow P_1 \rightarrow P_0 \rightarrow k \rightarrow 0\) be a projective resolution of the trivial \(kG\)-module \(k\), where \(P_0\) and \(P_1\) are finitely generated free \(kG\)-modules. Let \(D \subseteq \mathcal D_1\) be the domain of \(\rho\). Since projective modules over local rings are free by Kaplansky's Theorem \cite{Kaplansky_ProjectiveModules}, there is some cardinal \(\alpha\) such that \(D \otimes_{kG} P_2 \cong D^{\oplus\alpha}\) as left \(D\)-modules, and therefore
    \[
        \mathcal D_i \otimes_{kG} P_2 \cong \mathcal D_i \otimes_D D \otimes_{kG} \otimes P_i \cong \mathcal D_i^{\oplus \alpha}
    \]
    for \(i = 1,2\). The assumptions that \(b_2(G;\mathcal D_2) < \infty\) and that \(P_1\) is finitely generated imply that \(\alpha < \infty\).

    Since \(\alpha < \infty\), there is a finitely generated free \(kG\)-module \(F\) and a homomorphism \(\partial \colon F \rightarrow P_2\) such that  
    \[
        \mathcal D_2 \otimes_{kG} F \rightarrow \mathcal D_2 \otimes_{kG} P_2 \cong \mathcal D_2^\alpha
    \]
    is an epimorphism. By passing to a free submodule of \(F\), we may assume that the above map is an isomorphism. For \(i = 1,2\), the maps \(\mathcal D_i \otimes_{kG} F \rightarrow \mathcal D_i \otimes_{kG} P_2\) are obtained by applying \(\mathcal D_i \otimes_D -\) to the map of finitely generated free \(D\)-modules \(D \otimes_{kG} F \rightarrow D \otimes_{kG} P_2\). But \(\rho\) is also a specialisation of \(D\)-division rings, so the rank functions still satisfy \(\rk_{\mathcal D_1} \geqslant \rk_{\mathcal D_2}\) when viewed as rank functions on \(D\) by \cref{thm:Malcolmson}. It follows that 
    \[
        \mathcal D_1 \otimes_{kG} F \rightarrow \mathcal D_1 \otimes_{kG} P_2 \cong \mathcal D_1^\alpha
    \]
    is also an isomorphism, and therefore that the chain complexes
    \[
        0 \rightarrow \mathcal D_i \otimes_{kG} F \rightarrow \mathcal D_i \otimes_{kG} P_1 \rightarrow \mathcal D_i \otimes_{kG} P_0 \rightarrow 0
    \]
    compute the Betti numbers \(b_n(G;\mathcal D_i)\) for both \(i = 1,2\). Once again using the fact that \(\rk_{\mathcal D_1} \geqslant \rk_{\mathcal D_2}\), we have
    \[
        b_2(G; \mathcal D_1) = \alpha - \rk_{\mathcal D_1} \partial \leqslant \alpha - \rk_{\mathcal D_2} \partial = b_2(G; \mathcal D_2). \qedhere
    \]
\end{proof}

As a consequence, we obtain the following criterion for finite presentability of a \(2\)-dimensional group, which may be of independent interest.

\begin{cor}\label{cor:bZero_b2Zero}
    Let \(G\) be a finitely generated RPVN group of \(\cd_k(G) \leqslant 2\) for some field \(k\). If \(b_2(G;k) = 0\), then \(G\) is coherent and, in particular, is finitely presented. More generally, if there is a subgroup \(H \leqslant G\) of finite index such that \(b_2(H;k) < [G:H]\), then \(G\) is coherent.
\end{cor}
\begin{proof}
    If \(G\) is RPVN, then \(\Dk{G}\) is the universal division ring of \(kG\) \cite[Corollary 1.3]{JaikinZapirain2020THEUO}, so there is a specialisation \(\Dk{G} \rightarrow k\) of \(kG\)-division rings. Then 
    \[
        \btwo{2}(G;k) = \frac{b_2(H;k)}{[G:H]} \leqslant \frac{b_2(H;k)}{[G:H]} < 1
    \]
    by \cref{lem:removeFP2}. But \(\btwo{2}(G;k)\) is an integer, so \(\btwo{2}(G;k) = 0\) and the corollary then follows from \cref{lem:removeFP2} and \cref{cor:2dim_coherence}. \qedhere
\end{proof}

For a general RPVN group \(G\) of \(\cd_k(G) = d\), the Betti numbers \(\btwo{n}(G;k)\) may depend on the characteristic of \(k\). This even happens for RFRS groups as the following example shows. Let \(L\) be a flag complex which is homeomorphic to the space obtained by gluing the boundary of an \((n-1)\)-ball to \(S^{n-2}\) via a degree \(p\) map, for some prime \(p\). Then \(b_{n-1}(L;\Q) = 0\) but \(b_{n-1}(G; \mathbb{F}_p) = 1\). Let \(G\) be the RAAG on the one-skeleton of \(L\). Then \(\cd_\Q(G) = \cd_{\mathbb F_p}(G) = n\) and
\[
    0 = \btwo{n}(G;\Q) < \btwo{n}(G;\mathbb F_p) = 1.
\]
However, \cref{cor:2dim_coherence} shows that this phenomenon does not occur in dimension \(2\): if \(G\) is a finitely generated RPVN group of \(\cd(G) = 2\), then \(\btwo{2}(G;k)\) is either zero for all fields \(k\) or is positive for all fields \(k\). We thus obtain the following application to homology growth of \(2\)-dimensional groups in positive characteristic.

\begin{cor}
    Let \(G\) be a finitely generated RPVN group with \(\cd_\Z(G) \leqslant 2\). The quantity
    \[
        \beta_2(G;k) := \inf\left\{ \frac{b_2(H;k)}{[G:H]} \ : \ H \leqslant G, \ [G:H] < \infty \right\}
    \]
    is either zero for all fields \(k\) or it is positive for all fields \(k\).
\end{cor}
\begin{proof}
    Suppose \(\beta_2(G;k) = 0\) for some field \(k\). Then there is a finite-index subgroup \(H \leqslant G\) such that \(b_2(H;k) < [G:H]\). As in the proof of \cref{cor:bZero_b2Zero}, this implies that \(\btwo{2}(G;k) = 0\) and thus that \(G\) is free-by-(poly-\(\Z\)). But then \(\btwo{2}(G;k_0) = 0\) for all fields \(k_0\), by \cref{prop:divring_gaboriau}.
    
    Since \(G\) is finitely presented, this implies in particular that \(G\) is of type \(\FP(k)\). It then follows from \cite[Theorem 1.2]{JaikinZapirain2020THEUO} and \cite[Theorem 0.2]{LinnellLuckSauer_modpAmenableGrowth} that \(\beta_2(G;k_0) = \btwo{2}(G;k_0) = 0\) for all fields \(k_0\). \qedhere
\end{proof}

We can now prove \cref{cor:D}. Recall that a group \(G\) is \emph{parafree} if it is residually nilpotent and there is a free group \(F\) such that \(G/\gamma_n(G) \cong F/\gamma_n(F)\) for all \(n \in \N\). The \emph{Parafree Conjecture} for a finitely generated parafree group \(G\) predicts that \(\H_2(G;\Z) = 0\), and the \emph{Strong Parafree Conjecture} additionally predicts that \(\cd(G) \leqslant 2\). We thank Ismael Morales, who informed us that the vanishing of \(\btwo{2}(G)\) implies \(b_2(G) = 0\) for finitely generated parafree groups of cohomological dimension \(2\).

\begin{cor}
    Let \(G\) be a finitely generated parafree group of cohomological dimension \(2\). Then the Parafree Conjecture holds for \(G\) if and only if \(\gamma_n(G)\) is free for all sufficiently large \(n\). In particular, finitely generated parafree groups satisfying the Strong Parafree Conjecture are coherent.
\end{cor}
\begin{proof}
    If the Parafree Conjecture holds, then \(b_2(G) = 0\), and therefore \(\btwo{2}(G) = 0\) by \cref{lem:removeFP2}. By \cref{cor:RTFN}, \(\gamma_n(G)\) is then free for all sufficiently large \(n\).
    
    Conversely, if some \(\gamma_n(G)\) is free, then \(G\) is free-by-(free nilpotent) and therefore is coherent by \cref{cor:2dim_coherence} and in particular is finitely presented. Moreover, \(\btwo{1}(G) = b_1(G;k) - 1\) by \cite[Corollary 8.1]{BridsonReid_BaumslagProblems} and \(\btwo{2}(G;k) = 0\) for all fields \(k\). Hence, the Euler characteristic of \(G\) is
    \[
        \chi(G) = -\btwo{1}(G) = 1 - b_1(G;k) + b_2(G;k)
    \]
    for every field \(k\), which implies \(b_2(G;k) = 0\) for all \(k\). Hence, \(H_2(G;\Z) = 0\), so \(G\) satisfies the Parafree Conjecture.
\end{proof}

\begin{rem}
    Note that we can also conclude that finitely generated parafree groups satisfying the Strong Parafree Conjecture are virtually free-by-cyclic, and therefore coherent by \cite{FeighnHandel_FreeByZCoherent}. Indeed, if \(\H_2(G;\Z) = 0\) for a finitely generated parafree group of cohomological dimension \(2\), then \(\btwo{2}(G) = 0\) by \cref{lem:removeFP2}. Parafree groups are RFRS by \cite[Theorem 9.2]{Reid_parafreeRFRS}, so \(G\) is virtually free-by-cyclic by \cite[Theorem A]{Fisher_freebyZ}.
\end{rem}


\section{Questions}

We conclude with two natural questions raised by the results of the article. The first pertains to the closure properties of the class of RPVN groups.

\begin{q}
    Let \(\mathcal C\) be a class of locally indicable and amenable groups. If \(G\) is finitely generated residually \(\mathcal C\) and of cohomological dimension \(2\), does \(\btwo{2}(G) = 0\) imply that \(G\) is free-by-\(C\) for some \(C \in \mathcal C\)?
\end{q}

An interesting test case for this question is the class \(\mathcal C\) of poly-\(\Z\) groups. 

The methods of this article unfortunately do not allow us to obtain a result on algebraic fibring, which was the original aim of Kielak's description of \(\mathcal D_{\Q G}\) for \(G\) RFRS. If \(G\) is RPVN and \(\btwo{1}(G) = 0\), then we can produce a virtual map to a nilpotent group \(H \rightarrow Q\) and an `open set' of multicharacters on \(Q\) such that the \(\H_1(H;\nov{\Q H}{\pm\chi}) = 0\). In order to obtain a fibring theorem, we would need to perturb the multicharacter (in some appropriate sense) to obtain a character coming from a map to \(\Z\) which still has vanishing Novikov homology. When \(Q\) is Abelian, this is easy since open sets of characters always contain rational characters, but for \(Q\) nilpotent it is not clear to us what the correct notions of openness and perturbation are (or if they exist). The second author has developed a version of \(\Sigma\)-theory for groups mapping onto nilpotent quotients and formulated a criterion for the kernel of a  map onto a nilpotent group to be finitely generated \cite[Theorem A]{Klinge_nilpotentSigma}. One possibility would be to establish a connection between the vanishing of Novikov homology and these nilpotent quotient \(\Sigma\)-invariants. We conclude with the following question.

\begin{q}
    Let \(G\) be a finitely generated RPVN group with \(\btwo{1}(G) = 0\). Does \(G\) virtually algebraically fibre?
\end{q}

\bibliography{bib}

\begin{thebibliography}{JZL{\'{A}}20}

\bibitem[Ago08]{AgolCritVirtFib}
Ian Agol.
\newblock Criteria for virtual fibering.
\newblock {\em J. Topol.}, 1(2):269--284, 2008.

\bibitem[Bau67]{Baumslag_NonFreeParafree}
Gilbert Baumslag.
\newblock Groups with the same lower central sequence as a relatively free
  group. {I}. {T}he groups.
\newblock {\em Trans. Amer. Math. Soc.}, 129:308--321, 1967.

\bibitem[Bau99]{Baumslag_RTFN}
Gilbert Baumslag.
\newblock Finitely generated residually torsion-free nilpotent groups. {I}.
\newblock {\em J. Austral. Math. Soc. Ser. A}, 67(3):289--317, 1999.

\bibitem[Bau05]{Baumslag_parafreeSurvey}
Gilbert Baumslag.
\newblock Parafree groups.
\newblock In {\em Infinite groups: geometric, combinatorial and dynamical
  aspects}, volume 248 of {\em Progr. Math.}, pages 1--14. Birkh\"{a}user,
  Basel, 2005.

\bibitem[Bau10]{Baumslag_reflectionsRTFN}
Gilbert Baumslag.
\newblock Some reflections on proving groups residually torsion-free nilpotent.
  {I}.
\newblock {\em Illinois J. Math.}, 54(1):315--325, 2010.

\bibitem[BB09]{BardakovBellingeri_purebraidgroupsurfaceRTFN}
Valerij~G. Bardakov and Paolo Bellingeri.
\newblock On residual properties of pure braid groups of closed surfaces.
\newblock {\em Comm. Algebra}, 37(5):1481--1490, 2009.

\bibitem[BGR03]{BludovGlassRhem_centralConvex}
V.~V. Bludov, A.~M.~W. Glass, and Akbar~H. Rhemtulla.
\newblock Ordered groups in which all convex jumps are central.
\newblock {\em J. Korean Math. Soc.}, 40(2):225--239, 2003.

\bibitem[Bie81]{BieriQueenMary}
Robert Bieri.
\newblock {\em Homological dimension of discrete groups}.
\newblock Queen Mary College Mathematics Notes. Queen Mary College, Department
  of Pure Mathematics, London, second edition, 1981.

\bibitem[BR15]{BridsonReid_BaumslagProblems}
Martin~R. Bridson and Alan~W. Reid.
\newblock Nilpotent completions of groups, {G}rothendieck pairs, and four
  problems of {B}aumslag.
\newblock {\em Int. Math. Res. Not. IMRN}, (8):2111--2140, 2015.

\bibitem[Coc87]{Cochran_LinkConcordance}
Tim~D. Cochran.
\newblock Link concordance invariants and homotopy theory.
\newblock {\em Invent. Math.}, 90(3):635--645, 1987.

\bibitem[Coh85]{Cohn_FreeRingsRelations}
Paul~M. Cohn.
\newblock {\em Free rings and their relations}, volume~19 of {\em London
  Mathematical Society Monographs}.
\newblock Academic Press, Inc. [Harcourt Brace Jovanovich, Publishers], London,
  second edition, 1985.

\bibitem[DK92]{DuchampKrob_RAAGsRTFN}
G\'erard Duchamp and Daniel Krob.
\newblock The lower central series of the free partially commutative group.
\newblock {\em Semigroup Forum}, 45(3):385--394, 1992.

\bibitem[FH99]{FeighnHandel_FreeByZCoherent}
Mark Feighn and Michael Handel.
\newblock Mapping tori of free group automorphisms are coherent.
\newblock {\em Ann. of Math. (2)}, 149(3):1061--1077, 1999.

\bibitem[Fis21]{Fisher_Improved}
Sam~P. Fisher.
\newblock Improved algebraic fibrings, 2021.
\newblock {\tt arXiv:2112.00397}.

\bibitem[Fis24]{Fisher_freebyZ}
Sam~P. Fisher.
\newblock On the cohomological dimension of kernels of maps to \(\mathbb{Z}\),
  2024.
\newblock {\tt arXiv:2403.18758}.

\bibitem[FM23]{FisherMorales_HNC}
Sam~P. Fisher and Ismael Morales.
\newblock The {H}anna {N}eumann conjecture for graphs of free groups with
  cyclic edge groups, 2023.
\newblock {\tt arXiv:2311.12910}.

\bibitem[FR88]{FalkRandell_purebraidgroupsRTFN}
Michael Falk and Richard Randell.
\newblock Pure braid groups and products of free groups.
\newblock In {\em Braids ({S}anta {C}ruz, {CA}, 1986)}, volume~78 of {\em
  Contemp. Math.}, pages 217--228. Amer. Math. Soc., Providence, RI, 1988.

\bibitem[Gab02]{Gaboriau2002}
Damien Gaboriau.
\newblock Invariants {\(\ell^2\)} de relations d'\'{e}quivalence et de groupes.
\newblock {\em Publ. Math. Inst. Hautes \'{E}tudes Sci.}, 95:93--150, 2002.

\bibitem[Gr{\"a}20]{Grater20}
Joachim Gr{\"a}ter.
\newblock Free division rings of fractions of crossed products of groups with
  {C}onradian left-orders.
\newblock {\em Forum Math.}, 32(3):739--772, 2020.

\bibitem[Hug70]{HughesDivRings1970}
Ian Hughes.
\newblock Division rings of fractions for group rings.
\newblock {\em Comm. Pure Appl. Math.}, 23:181--188, 1970.

\bibitem[Hum77]{Humphreys_polycyclic}
J.~F. Humphreys.
\newblock Generator conditions on the {F}itting subgroup of a polycyclic group.
\newblock {\em Proc. Edinburgh Math. Soc. (2)}, 20(4):273--278, 1976/77.

\bibitem[HW08]{HaglundWise_special}
Fr\'{e}d\'{e}ric Haglund and Daniel~T. Wise.
\newblock Special cube complexes.
\newblock {\em Geom. Funct. Anal.}, 17(5):1551--1620, 2008.

\bibitem[JZ21]{JaikinZapirain2020THEUO}
Andrei Jaikin-Zapirain.
\newblock The universality of {H}ughes-free division rings.
\newblock {\em Selecta Math. (N.S.)}, 27(4):Paper No. 74, 33, 2021.

\bibitem[JZ24]{JaikinZapirain_freeQgroups}
Andrei Jaikin-Zapirain.
\newblock Free {$\mathbb{Q}$}-groups are residually torsion-free nilpotent.
\newblock {\em Ann. Sci. \'{E}c. Norm. Sup\'{e}r. (4)}, 57(4):1101--1136, 2024.

\bibitem[JZL23]{JaikinLinton_coherence}
Andrei Jaikin-Zapirain and Marco Linton.
\newblock On the coherence of one-relator groups and their group algebras,
  2023.
\newblock {\tt arXiv:2303.05976}.

\bibitem[JZL{\'{A}}20]{JaikinLopezStrongAtiyah2020}
Andrei Jaikin-Zapirain and Diego L{\'{o}}pez-{\'{A}}lvarez.
\newblock The strong {A}tiyah and {L}\"{u}ck approximation conjectures for
  one-relator groups.
\newblock {\em Math. Ann.}, 376(3-4):1741--1793, 2020.

\bibitem[JZM24]{JaikinMorales_surfaceProfRigid}
Andrei Jaikin-Zapirain and Ismael Morales.
\newblock Prosolvable rigidity of surface groups, 2024.
\newblock {\tt arXiv:2312.12293}.

\bibitem[Kap58]{Kaplansky_ProjectiveModules}
Irving Kaplansky.
\newblock Projective modules.
\newblock {\em Ann. of Math. (2)}, 68:372--377, 1958.

\bibitem[Kie20]{KielakRFRS}
Dawid Kielak.
\newblock Residually finite rationally solvable groups and virtual fibring.
\newblock {\em J. Amer. Math. Soc.}, 33(2):451--486, 2020.

\bibitem[Kli23]{Klinge_nilpotentSigma}
Kevin Klinge.
\newblock Sigma invariants for partial orders on nilpotent groups, 2023.
\newblock {\tt arXiv:2311.00620}.

\bibitem[KLM88]{KrophollerLinnellMoody_Ore}
P.~H. Kropholler, P.~A. Linnell, and J.~A. Moody.
\newblock Applications of a new {$K$}-theoretic theorem to soluble group rings.
\newblock {\em Proc. Amer. Math. Soc.}, 104(3):675--684, 1988.

\bibitem[Lin93]{LinnellDivRings93}
Peter~A. Linnell.
\newblock Division rings and group von {N}eumann algebras.
\newblock {\em Forum Math.}, 5(6):561--576, 1993.

\bibitem[LLS11]{LinnellLuckSauer_modpAmenableGrowth}
Peter Linnell, Wolfgang L\"{u}ck, and Roman Sauer.
\newblock The limit of {$\mathbb F_p$}-{B}etti numbers of a tower of finite
  covers with amenable fundamental groups.
\newblock {\em Proc. Amer. Math. Soc.}, 139(2):421--434, 2011.

\bibitem[L{\"u}c02]{Luck02}
Wolfgang L{\"u}ck.
\newblock {\em {$L\sp 2$}-invariants: theory and applications to geometry and
  {$K$}-theory}.
\newblock Springer-Verlag, Berlin, 2002.

\bibitem[Mag35]{Magnus_freegrpsRTFN}
Wilhelm Magnus.
\newblock Beziehungen zwischen {G}ruppen und {I}dealen in einem speziellen
  {R}ing.
\newblock {\em Math. Ann.}, 111(1):259--280, 1935.

\bibitem[Mal48]{Malcev_series}
Anatoli\u\i~I. Mal{'}cev.
\newblock On the embedding of group algebras in division algebras.
\newblock {\em Doklady Akad. Nauk SSSR (N.S.)}, 60:1499--1501, 1948.

\bibitem[Mal80]{Malcolmson_SkewFields}
Peter Malcolmson.
\newblock Determining homomorphisms to skew fields.
\newblock {\em J. Algebra}, 64(2):399--413, 1980.

\bibitem[Neu49]{Neumann_series}
Bernhard~H. Neumann.
\newblock On ordered division rings.
\newblock {\em Trans. Amer. Math. Soc.}, 66:202--252, 1949.

\bibitem[OS24]{OkunSchreve_DawidSimplified}
Boris Okun and Kevin Schreve.
\newblock Orders and fibering, 2024.
\newblock {\tt arXiv:2403.16102}.

\bibitem[Pas77]{PassmanGrpRng}
Donald~S. Passman.
\newblock {\em The algebraic structure of group rings}.
\newblock Pure and Applied Mathematics. Wiley-Interscience [John Wiley \&
  Sons], New York-London-Sydney, 1977.

\bibitem[Rei15]{Reid_parafreeRFRS}
Alan~W. Reid.
\newblock Profinite properties of discrete groups.
\newblock In {\em Groups {S}t {A}ndrews 2013}, volume 422 of {\em London Math.
  Soc. Lecture Note Ser.}, pages 73--104. Cambridge Univ. Press, Cambridge,
  2015.

\bibitem[Swa69]{Swan_cd1}
Richard~G. Swan.
\newblock Groups of cohomological dimension one.
\newblock {\em J. Algebra}, 12:585--610, 1969.

\bibitem[Wis20]{Wise_anInvitation}
Daniel~T. Wise.
\newblock An invitation to coherent groups.
\newblock In {\em What's next?---the mathematical legacy of {W}illiam {P}.
  {T}hurston}, volume 205 of {\em Ann. of Math. Stud.}, pages 326--414.
  Princeton Univ. Press, Princeton, NJ, 2020.

\end{thebibliography}
\bibliographystyle{alpha}

\end{document}